\documentclass{amsart}
\usepackage[foot]{amsaddr}
\usepackage{graphicx}

\usepackage[hang,flushmargin]{footmisc}
\usepackage[utf8]{inputenc}
\usepackage{amsmath, amsthm, amssymb}
\usepackage[usenames,dvipsnames,svgnames,table]{xcolor}
\usepackage[margin=1.3in]{geometry}
\usepackage{enumitem}
\usepackage{cite}
\usepackage[unicode,colorlinks=true,linkcolor=blue,citecolor=blue,bookmarks=true,bookmarksnumbered]{hyperref}
\usepackage{cleveref}
\usepackage{esint,dsfont}
\usepackage{xcolor}
\usepackage[normalem]{ulem}
\usepackage{verbatim}
\usepackage{xfrac}
\renewcommand{\phi}{\varphi}

\usepackage{tikz, pgfplots}

\usepackage{etoolbox}
\cslet{blx@noerroretextools}\empty

\newtheorem{theorem}{Theorem}[section]
\newtheorem{proposition}[theorem]{Proposition}
\newtheorem{lemma}[theorem]{Lemma}
\newtheorem{definition}[theorem]{Definition}

\theoremstyle{definition}

\newcommand{\speedplots}{
	\begin{tikzpicture}
		\node[anchor=south west,inner sep=0] (image) at (0,0) {\includegraphics[height = 60pt, width=0.45\textwidth]{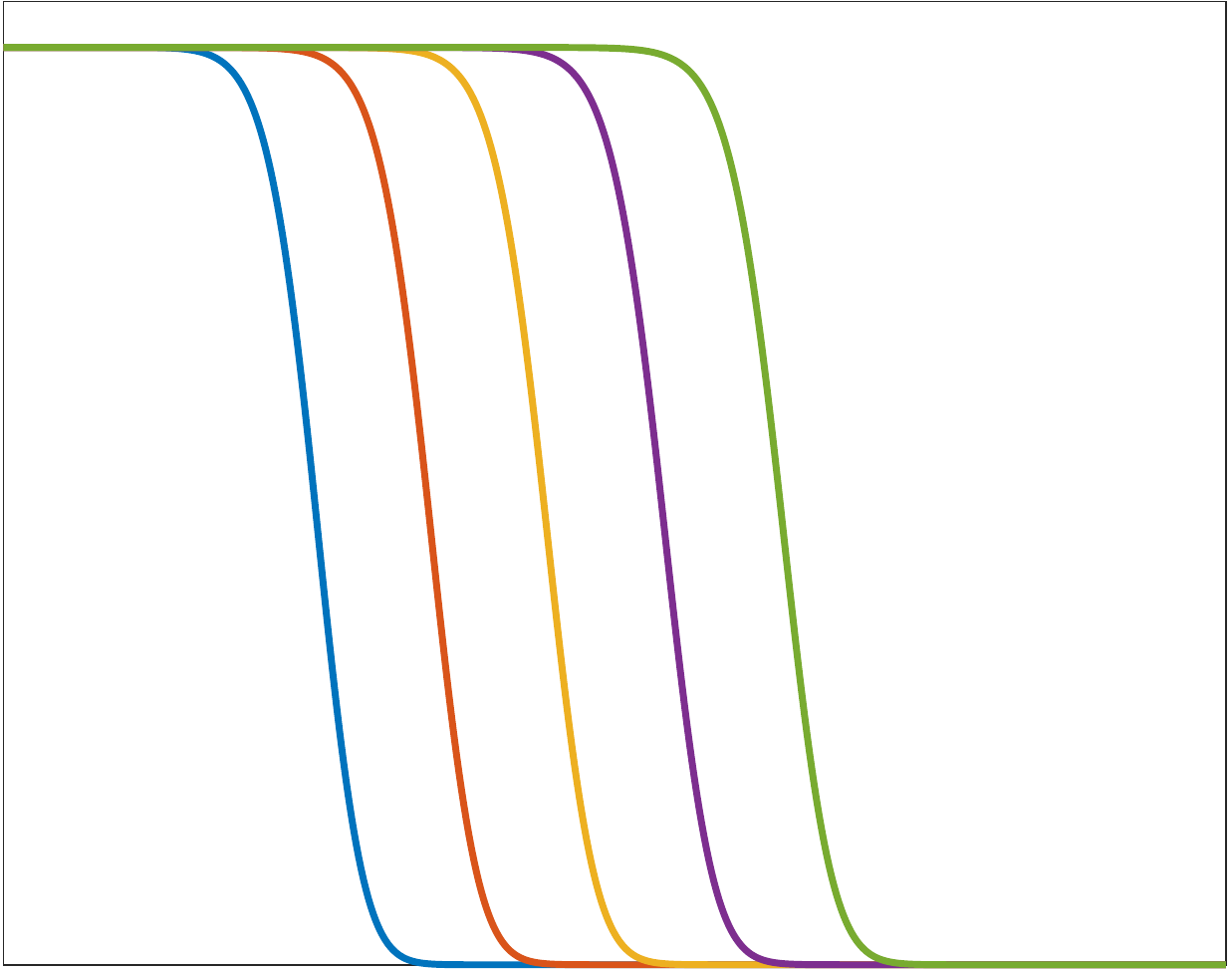}};
		\begin{scope}[x={(image.south east)},y={(image.north west)}]	    		
	\draw [thick,white, fill = white] (.98,0) rectangle (1,1);
    	\draw [thick,white, fill = white] (0,.98) rectangle (1,1);
    	\draw [thick,white, fill = white] (0,0) rectangle (.0032,1);
    	\draw [thick] (0,0) -- (1,0);
    	\draw [thick] (0,0) -- (0,1);
	
		\node [anchor=north east] at (0,0) {\tiny 0};
        \node [anchor=north] at (.5,0) {\tiny 50};
        \node [anchor=north] at (.975,0) {\tiny 100};
        \node [anchor=north west, rotate=-25] at (.254509,0) {\tiny\textbf{25.45}};
        \node [anchor=north west, rotate=-25] at (.635271,0) {\tiny\textbf{63.53}};
        \node [anchor=east] at (0,.475) {\tiny .5};
        \node [anchor=east] at (0,.95) {\tiny 1};
        \draw[thick] (.100,.02) -- (.1,-0.02);
        \draw[thick] (.200,.02) -- (.2,-0.02);
        \draw[thick] (.300,.02) -- (.3,-0.02);
        \draw[thick] (.400,.02) -- (.4,-0.02);
        \draw[thick] (.500,.02) -- (.5,-0.02);
        \draw[thick] (.600,.02) -- (.6,-0.02);
        \draw[thick] (.700,.02) -- (.7,-0.02);
        \draw[thick] (.800,.02) -- (.8,-0.02);
        \draw[thick] (.900,.02) -- (.9,-0.02);
        \draw[thick] (.995,.02) -- (.995,-0.02);
        \draw[thick] (-.01,.2375) -- (0.01,.2375);
        \draw[thick] (-.01,.475) -- (0.01,.475);
        \draw[thick] (-.01,.7125) -- (0.01,.7125);
        \draw[thick] (-.01,.9500) -- (0.01,.9500);
        \draw[thick,dotted] (0.01,.475) -- (.635271,.475);
        \draw[thick,dotted] (.254509,0.01) -- (.254509,.475);
        \draw[thick,dotted] (.346693,0.01) -- (.346693,.475);
        \draw[thick,dotted] (.442886,0.01) -- (.442886,.475);
        \draw[thick,dotted] (.539078,0.01) -- (.539078,.475);
        \draw[thick,dotted] (.635271,0.01) -- (.635271,.475);
        
        \node at (.9,.85) {$\chi = 1$};
    \end{scope}
\end{tikzpicture}
\
\begin{tikzpicture}
    \node[anchor=south west,inner sep=0] (image) at (0,0) {\includegraphics[height = 60pt,width=0.45\textwidth]{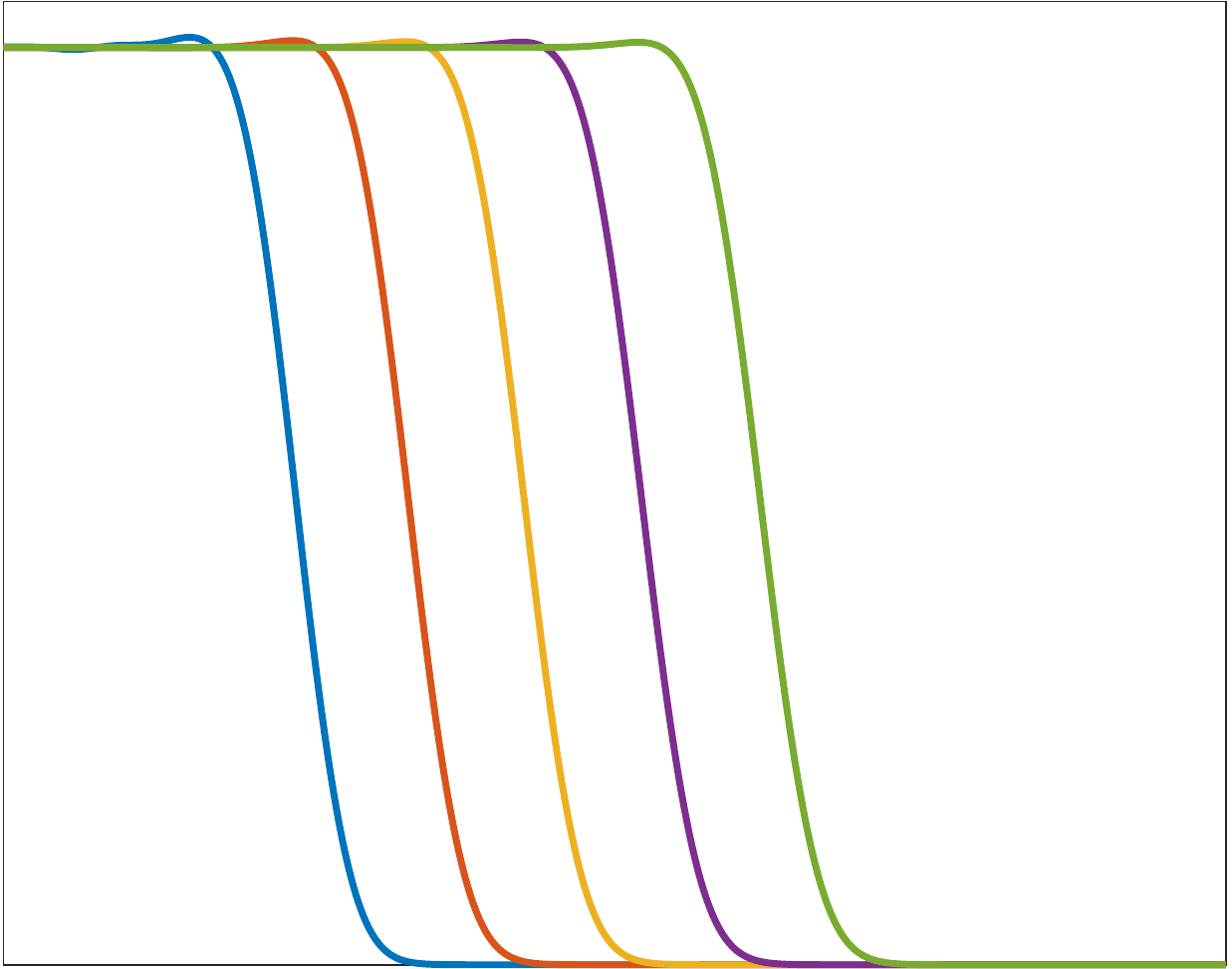}};
    \begin{scope}[x={(image.south east)},y={(image.north west)}]

	\draw [thick,white, fill = white] (.98,0) rectangle (1,1);
    	\draw [thick,white, fill = white] (0,.98) rectangle (1,1);
    	\draw [thick,white, fill = white] (0,0) rectangle (.0032,1);
    	\draw [thick] (0,0) -- (1,0);
    	\draw [thick] (0,0) -- (0,1);

        \node [anchor=north east] at (0,0) {\tiny0};
        \node [anchor=north] at (.5,0) {\tiny50};
        \node [anchor=north] at (.975,0) {\tiny100};
        \node [anchor=north west, rotate=-25] at (.240481,0) {\tiny\textbf{24.05}};
        \node [anchor=north west, rotate=-25] at (.619238,0) {\tiny \textbf{61.92}};
        \node [anchor=east] at (0,.475) {\tiny.5};
        \node [anchor=east] at (0,.95) {\tiny 1};
        \draw[thick] (.100,.02) -- (.1,-0.02);
        \draw[thick] (.200,.02) -- (.2,-0.02);
        \draw[thick] (.300,.02) -- (.3,-0.02);
        \draw[thick] (.400,.02) -- (.4,-0.02);
        \draw[thick] (.500,.02) -- (.5,-0.02);
        \draw[thick] (.600,.02) -- (.6,-0.02);
        \draw[thick] (.700,.02) -- (.7,-0.02);
        \draw[thick] (.800,.02) -- (.8,-0.02);
        \draw[thick] (.900,.02) -- (.9,-0.02);
        \draw[thick] (.995,.02) -- (.995,-0.02);
        \draw[thick] (-.01,.2375) -- (0.01,.2375);
        \draw[thick] (-.01,.475) -- (0.01,.475);
        \draw[thick] (-.01,.7125) -- (0.01,.7125);
        \draw[thick] (-.01,.9500) -- (0.01,.9500);
        \draw[thick,dotted] (0.01,.475) -- (.619238,.475);
        \draw[thick,dotted] (.240481,0.01) -- (.240481,.475);
         \draw[thick,dotted] (.332665,0.01) -- (.332665,.475);
         \draw[thick,dotted] (.426854,0.01) -- (.426854,.475);
         \draw[thick,dotted] (.523046,0.01) -- (.523046,.475);
        \draw[thick,dotted] (.619238,0.01) -- (.619238,.475);
        \node at (.9, .85) {$\chi = 3$};
    \end{scope}
\end{tikzpicture}

\begin{tikzpicture}
    \node[anchor=south west,inner sep=0] (image) at (0,0) {\includegraphics[height = 60pt,width=0.45\textwidth]{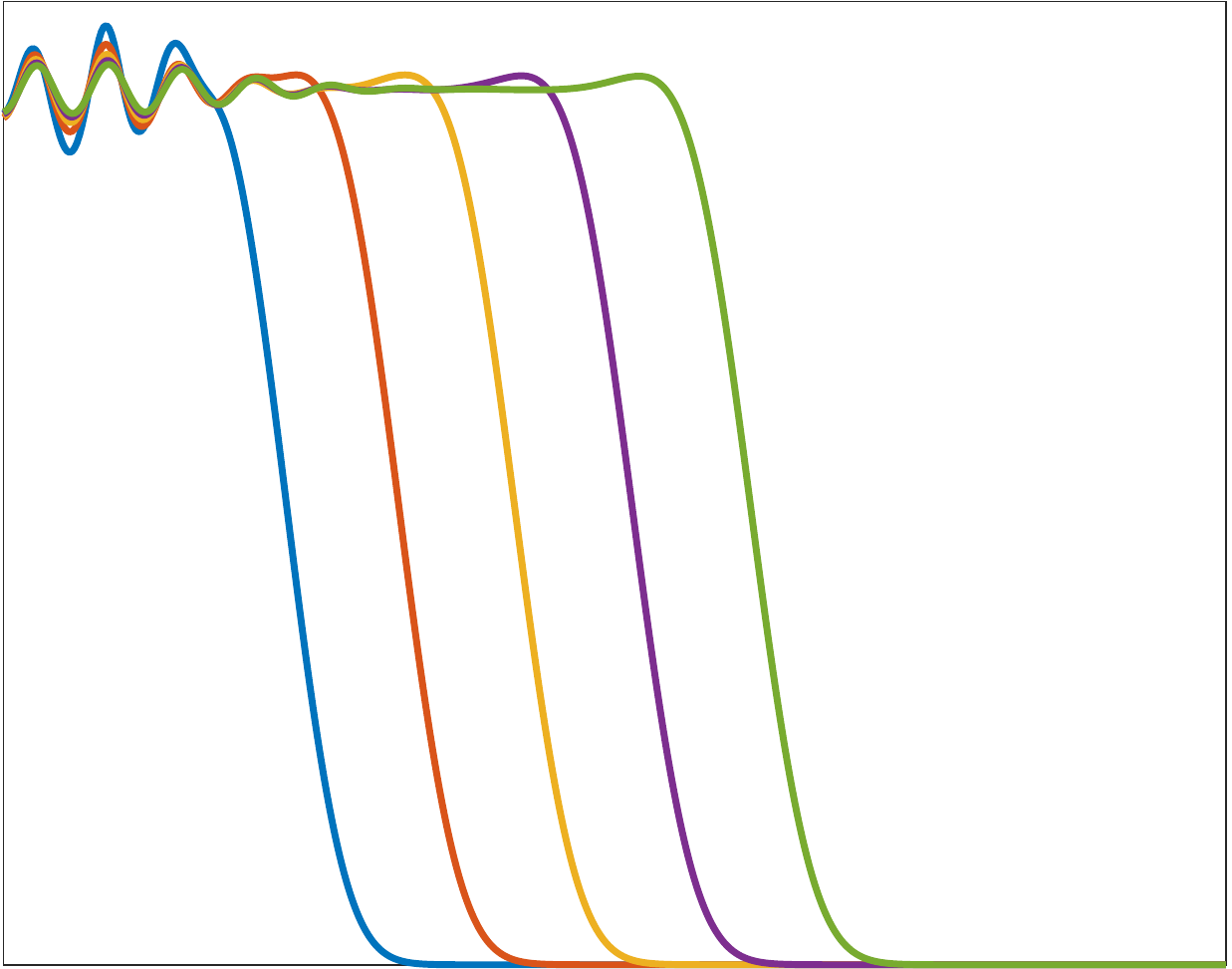}};
    \begin{scope}[x={(image.south east)},y={(image.north west)}]

	\draw [thick,white, fill = white] (.98,0) rectangle (1,1);
    	\draw [thick,white, fill = white] (0,.98) rectangle (1,1);
    	\draw [thick,white, fill = white] (0,0) rectangle (.0032,1);
    	\draw [thick] (0,0) -- (1,0);
    	\draw [thick] (0,0) -- (0,1);

        \node [anchor=north east] at (0,0) {\tiny 0};
        \node [anchor=north] at (.5,0) {\tiny50};
        \node [anchor=north] at (.975,0) {\tiny100};
        \node [anchor=north west, rotate=-25] at (.234469,0) {\tiny \textbf{23.45}};
        \node [anchor=north west, rotate=-25] at (.613226,0) {\textbf{\tiny 61.32}};
        \node [anchor=east] at (0,.45) {\tiny .5};
        \node [anchor=east] at (0,.90) {\tiny 1};
        \draw[thick] (.100,.02) -- (.1,-0.02);
        \draw[thick] (.200,.02) -- (.2,-0.02);
        \draw[thick] (.300,.02) -- (.3,-0.02);
        \draw[thick] (.400,.02) -- (.4,-0.02);
        \draw[thick] (.500,.02) -- (.5,-0.02);
        \draw[thick] (.600,.02) -- (.6,-0.02);
        \draw[thick] (.700,.02) -- (.7,-0.02);
        \draw[thick] (.800,.02) -- (.8,-0.02);
        \draw[thick] (.900,.02) -- (.9,-0.02);
        \draw[thick] (.995,.02) -- (.995,-0.02);
        \draw[thick] (-.01,.225) -- (0.01,.225);
        \draw[thick] (-.01,.450) -- (0.01,.450);
        \draw[thick] (-.01,.675) -- (0.01,.675);
        \draw[thick] (-.01,.900) -- (0.01,.900);
        \draw[thick,dotted] (0.01,.45) -- (.613226,.45);
        \draw[thick,dotted] (.234469,0.01) -- (.234469,.45);
        \draw[thick,dotted] (.324649,0.01) -- (.324649,.45);
        \draw[thick,dotted] (.420842,0.01) -- (.420842,.45);
        \draw[thick,dotted] (.517034,0.01) -- (.517034,.45);
        \draw[thick,dotted] (.613226,0.01) -- (.613226,.45);
         \node at (.9, .85) {$\chi = 4$};
    \end{scope}
\end{tikzpicture}
\
\begin{tikzpicture}
    \node[anchor=south west,inner sep=0] (image) at (0,0) {\includegraphics[height = 60 pt, width=0.45\textwidth]{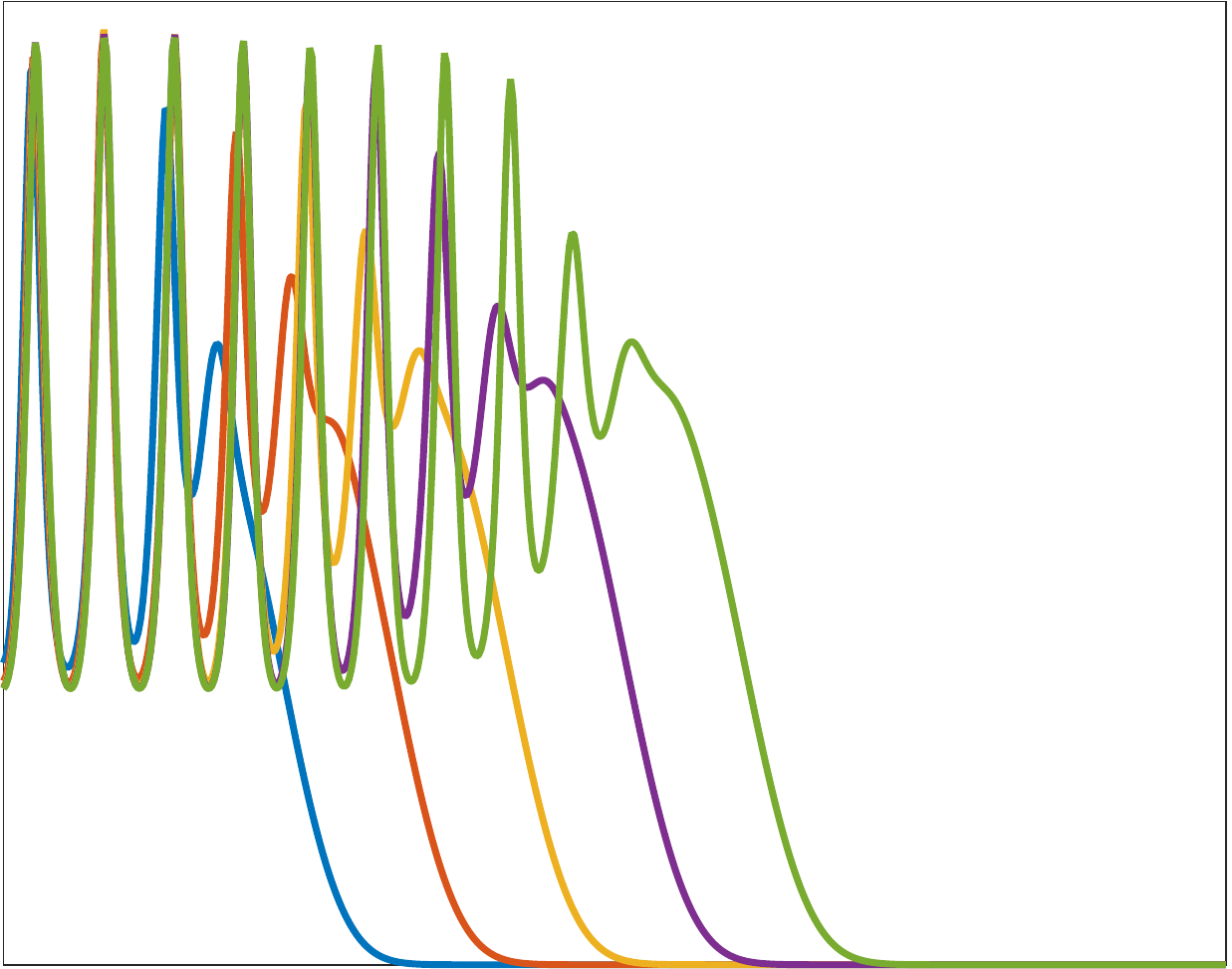}};
    \begin{scope}[x={(image.south east)},y={(image.north west)}]

	\draw [thick,white, fill = white] (.98,0) rectangle (1,1);
    	\draw [thick,white, fill = white] (0,.98) rectangle (1,1);
    	\draw [thick,white, fill = white] (0,0) rectangle (.0032,1);
    	\draw [thick] (0,0) -- (1,0);
    	\draw [thick] (0,0) -- (0,1);

        \node [anchor=north east] at (0,0) {\tiny0};
        \node [anchor=north] at (.5,0) {\tiny50};
        \node [anchor=north] at (.975,0) {\tiny100};
        \node [anchor=north west, rotate=-25] at (.238477,0) {\tiny\textbf{23.85}};
        \node [anchor=north west, rotate=-25] at (.617234,0) {\tiny\textbf{61.72}};
        \node [anchor=east] at (0,.475) {\tiny.8};
        \node [anchor=east] at (.01,.95) {\tiny1.6};
        \draw[thick] (.100,.02) -- (.1,-.02);
        \draw[thick] (.200,.02) -- (.2,-.02);
        \draw[thick] (.300,.02) -- (.3,-.02);
        \draw[thick] (.400,.02) -- (.4,-.02);
        \draw[thick] (.500,.02) -- (.5,-.02);
        \draw[thick] (.600,.02) -- (.6,-.02);
        \draw[thick] (.700,.02) -- (.7,-.02);
        \draw[thick] (.800,.02) -- (.8,-.02);
        \draw[thick] (.900,.02) -- (.9,-.02);
        \draw[thick] (.995,.02) -- (.995,-.02);
        \draw[thick] (-.01,.2375) -- (0.01,.2375);
        \draw[thick] (-.01,.4750) -- (0.01,.4750);
        \draw[thick] (-.01,.7125) -- (0.01,.7125);
        \draw[thick] (-.01,.9500) -- (0.01,.9500);
        \draw[thick,dotted] (0.01,.2375) -- (.617234,.2375);
        \draw[thick,dotted] (.238477,0.01) -- (.238477,.2375);
        \draw[thick,dotted] (.330661,0.01) -- (.330661,.2375);
        \draw[thick,dotted] (.424850,0.01) -- (.424850,.2375);
        \draw[thick,dotted] (.521042,0.01) -- (.521042,.2375);
        \draw[thick,dotted] (.617234,0.01) -- (.617234,.2375);
                 \node at (.9, .85) {$\chi = 5$};
    \end{scope}
\end{tikzpicture}
}

\newcommand{\patternplots}{
	\begin{tikzpicture}
		\node[anchor=south west,inner sep=0] (image) at (0,0) {\includegraphics[height = 60pt, width=0.45\textwidth]{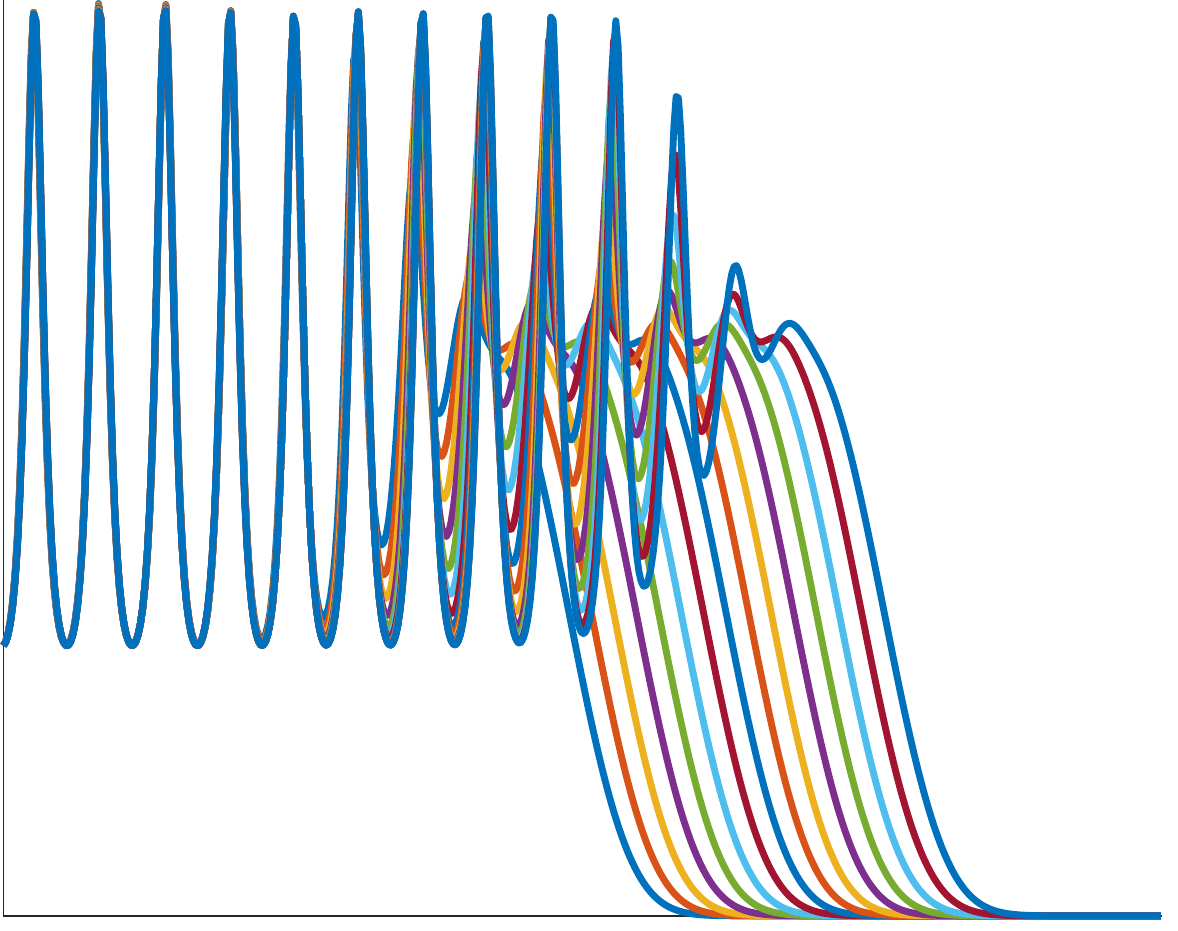}};
		\begin{scope}[x={(image.south east)},y={(image.north west)}]	    		
    	\draw [thick] (.002,.01) -- (1,.01);
    	\draw [thick] (.002,.01) -- (.002,1);

        \node [anchor=north east] at (.01,.05) {\tiny0};
        \node [anchor=north] at (.5,.02) {\tiny50};
        \node [anchor=north] at (.975,.02) {\tiny100};
        \node [anchor=east] at (0,.4925) {\tiny.8};
        \node [anchor=east] at (.01,.985) {\tiny1.6};
        \draw[thick] (.100,.03) -- (.1,-.01);
        \draw[thick] (.200,.03) -- (.2,-.01);
        \draw[thick] (.300,.03) -- (.3,-.01);
        \draw[thick] (.400,.03) -- (.4,-.01);
        \draw[thick] (.500,.03) -- (.5,-.01);
        \draw[thick] (.600,.03) -- (.6,-.01);
        \draw[thick] (.700,.03) -- (.7,-.01);
        \draw[thick] (.800,.03) -- (.8,-.01);
        \draw[thick] (.900,.03) -- (.9,-.01);
        \draw[thick] (.995,.03) -- (.995,-.01);
        \draw[thick] (-.008,.24625) -- (0.012,.24625);
        \draw[thick] (-.008,.49250) -- (0.012,.4925);
        \draw[thick] (-.008,.73875) -- (0.012,.73875);
        \draw[thick] (-.008,.985) -- (0.012,.985);
    \end{scope}
\end{tikzpicture}
\begin{tikzpicture}
    \node[anchor=south west,inner sep=0] (image) at (0,0) {\includegraphics[height = 60pt,width=0.45\textwidth]{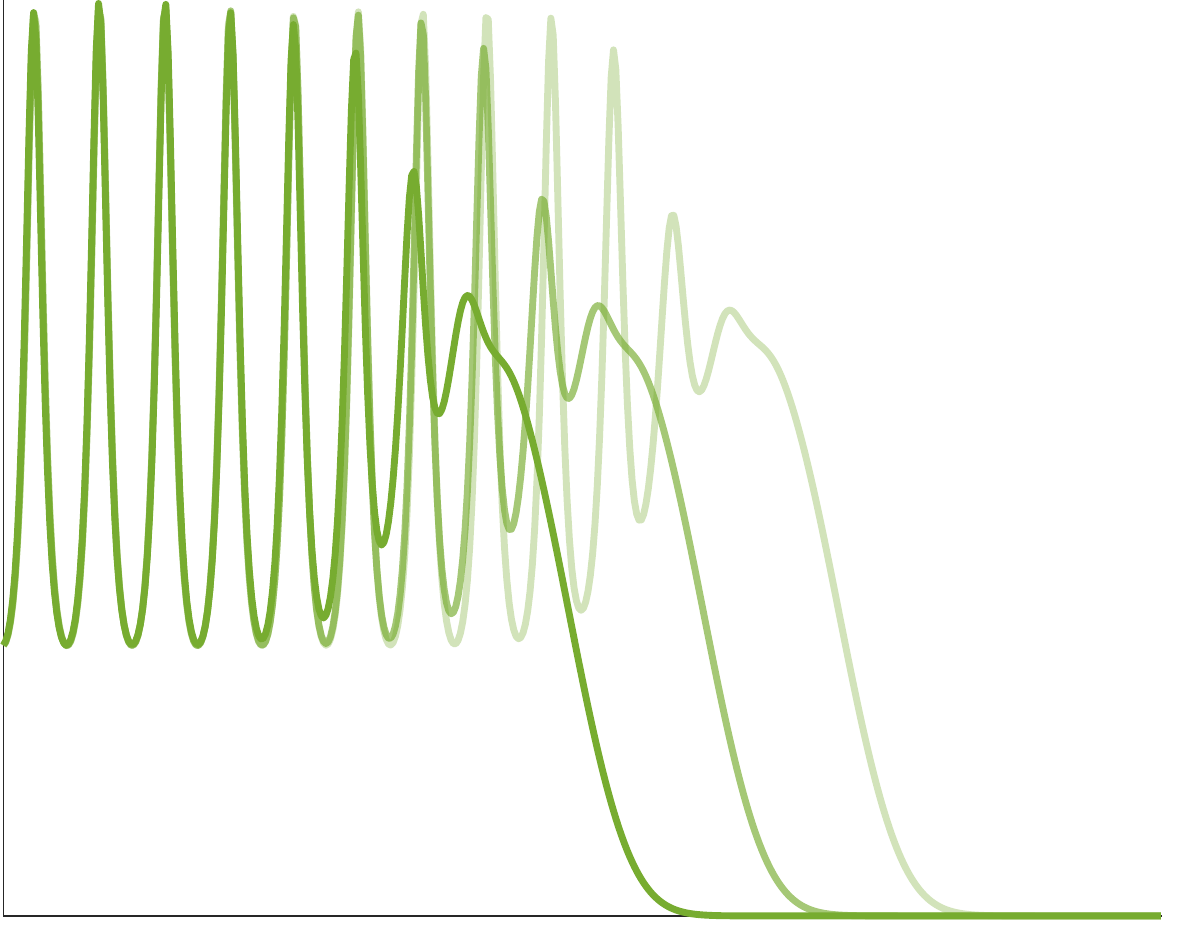}};
    \begin{scope}[x={(image.south east)},y={(image.north west)}]

    	\draw [thick] (.002,.01) -- (1,.01);
    	\draw [thick] (.002,.01) -- (.002,1);

       \node [anchor=north east] at (.01,.05) {\tiny0};
        \node [anchor=north] at (.5,.02) {\tiny50};
        \node [anchor=north] at (.975,.02) {\tiny100};
        \node [anchor=east] at (0,.4925) {\tiny.8};
        \node [anchor=east] at (.01,.985) {\tiny1.6};
        \draw[thick] (.100,.03) -- (.1,-.01);
        \draw[thick] (.200,.03) -- (.2,-.01);
        \draw[thick] (.300,.03) -- (.3,-.01);
        \draw[thick] (.400,.03) -- (.4,-.01);
        \draw[thick] (.500,.03) -- (.5,-.01);
        \draw[thick] (.600,.03) -- (.6,-.01);
        \draw[thick] (.700,.03) -- (.7,-.01);
        \draw[thick] (.800,.03) -- (.8,-.01);
        \draw[thick] (.900,.03) -- (.9,-.01);
        \draw[thick] (.995,.03) -- (.995,-.01);
        \draw[thick] (-.008,.24625) -- (0.012,.24625);
        \draw[thick] (-.008,.49250) -- (0.012,.4925);
        \draw[thick] (-.008,.73875) -- (0.012,.73875);
        \draw[thick] (-.008,.985) -- (0.012,.985);
    \end{scope}
\end{tikzpicture}

\begin{tikzpicture}
    \node[anchor=south west,inner sep=0] (image) at (0,0) {\includegraphics[height = 60pt,width=0.45\textwidth]{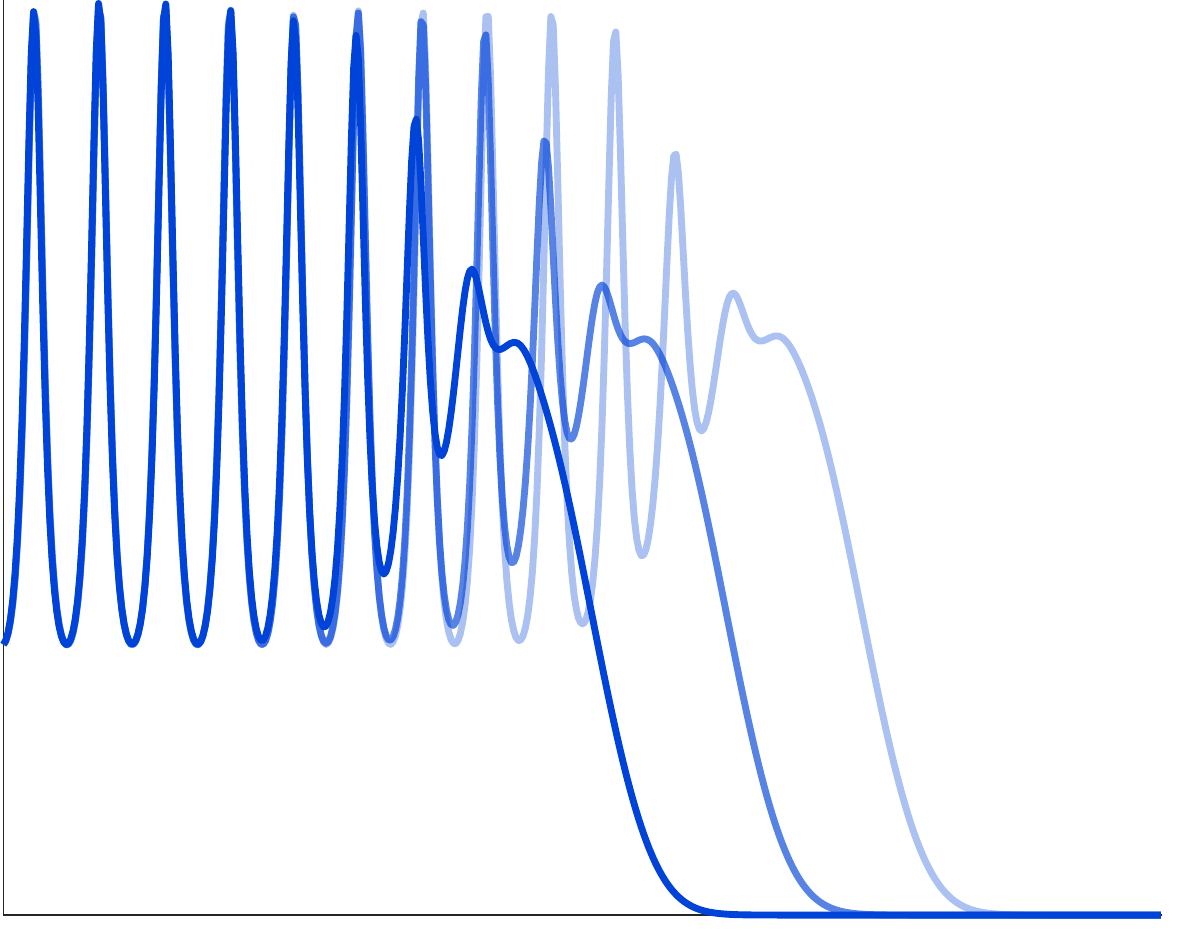}};
    \begin{scope}[x={(image.south east)},y={(image.north west)}]

    	\draw [thick] (.002,.01) -- (1,.01);
    	\draw [thick] (.002,.01) -- (.002,1);

       \node [anchor=north east] at (.01,.05) {\tiny0};
        \node [anchor=north] at (.5,.02) {\tiny50};
        \node [anchor=north] at (.975,.02) {\tiny100};
        \node [anchor=east] at (0,.4925) {\tiny.8};
        \node [anchor=east] at (.01,.985) {\tiny1.6};
        \draw[thick] (.100,.03) -- (.1,-.01);
        \draw[thick] (.200,.03) -- (.2,-.01);
        \draw[thick] (.300,.03) -- (.3,-.01);
        \draw[thick] (.400,.03) -- (.4,-.01);
        \draw[thick] (.500,.03) -- (.5,-.01);
        \draw[thick] (.600,.03) -- (.6,-.01);
        \draw[thick] (.700,.03) -- (.7,-.01);
        \draw[thick] (.800,.03) -- (.8,-.01);
        \draw[thick] (.900,.03) -- (.9,-.01);
        \draw[thick] (.995,.03) -- (.995,-.01);
        \draw[thick] (-.008,.24625) -- (0.012,.24625);
        \draw[thick] (-.008,.49250) -- (0.012,.4925);
        \draw[thick] (-.008,.73875) -- (0.012,.73875);
        \draw[thick] (-.008,.985) -- (0.012,.985);
    \end{scope}
\end{tikzpicture}
\begin{tikzpicture}
    \node[anchor=south west,inner sep=0] (image) at (0,0) {\includegraphics[height = 60 pt, width=0.45\textwidth]{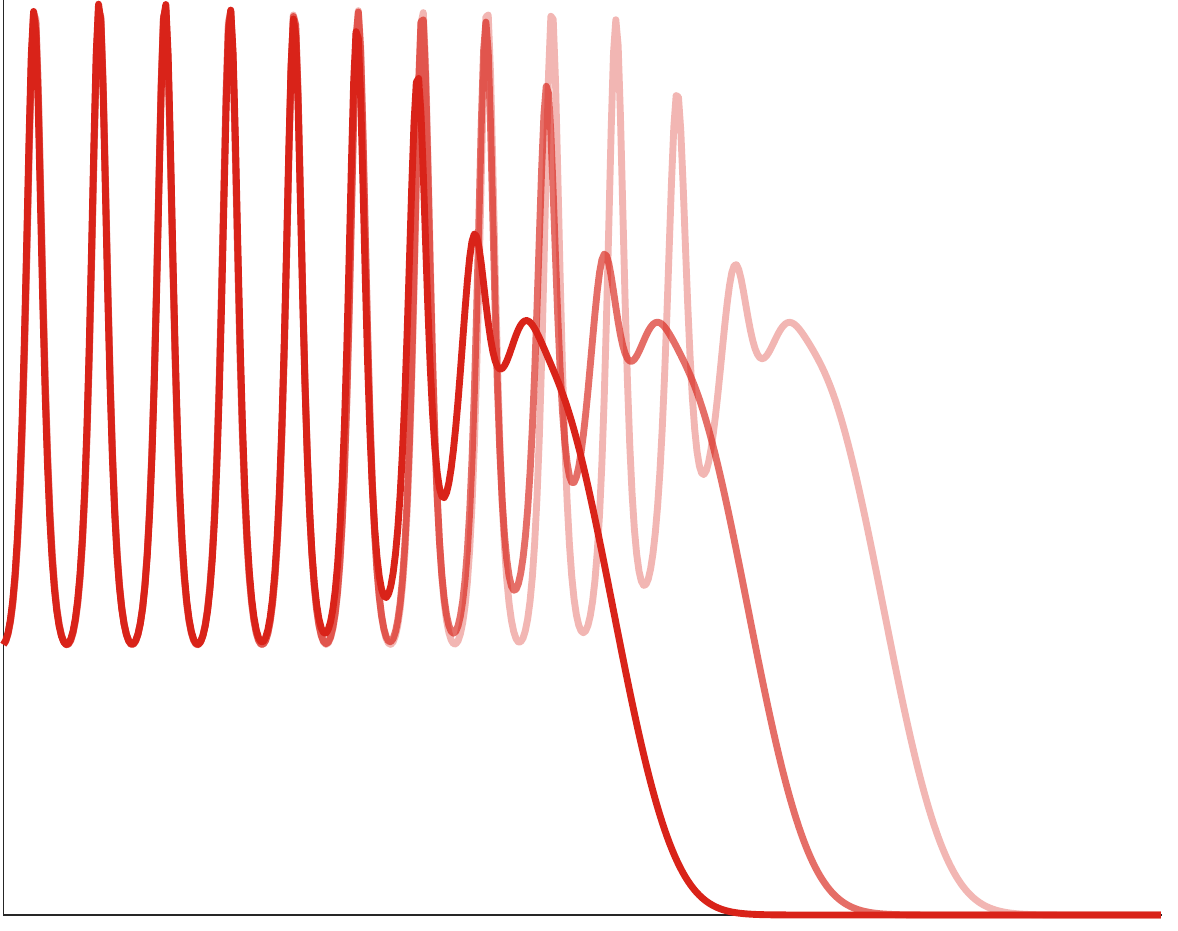}};
    \begin{scope}[x={(image.south east)},y={(image.north west)}]

    	\draw [thick] (.002,.01) -- (1,.01);
    	\draw [thick] (.002,.01) -- (.002,1);

       \node [anchor=north east] at (.01,.05) {\tiny0};
        \node [anchor=north] at (.5,.02) {\tiny50};
        \node [anchor=north] at (.975,.02) {\tiny100};
        \node [anchor=east] at (0,.4925) {\tiny.8};
        \node [anchor=east] at (.01,.985) {\tiny1.6};
        \draw[thick] (.100,.03) -- (.1,-.01);
        \draw[thick] (.200,.03) -- (.2,-.01);
        \draw[thick] (.300,.03) -- (.3,-.01);
        \draw[thick] (.400,.03) -- (.4,-.01);
        \draw[thick] (.500,.03) -- (.5,-.01);
        \draw[thick] (.600,.03) -- (.6,-.01);
        \draw[thick] (.700,.03) -- (.7,-.01);
        \draw[thick] (.800,.03) -- (.8,-.01);
        \draw[thick] (.900,.03) -- (.9,-.01);
        \draw[thick] (.995,.03) -- (.995,-.01);
        \draw[thick] (-.008,.24625) -- (0.012,.24625);
        \draw[thick] (-.008,.49250) -- (0.012,.4925);
        \draw[thick] (-.008,.73875) -- (0.012,.73875);
        \draw[thick] (-.008,.985) -- (0.012,.985);
    \end{scope}
\end{tikzpicture}
}

\newsavebox{\accentbox}

\newcommand{\R}{\mathbb{R}}

\newcommand{\calA}{\mathcal{A}}
\newcommand{\calB}{\mathcal{B}}

\newcommand{\calF}{\mathcal{F}}

\newcommand{\eul}{\sigma_{\rm ul}}
\newcommand{\lul}{L_{\rm ul}}

\newcommand{\va}{V_a}
\newcommand{\coloneqq}{:=}
\newcommand{\eps}{\varepsilon}
\newcommand{\e}{\eps}

\newcommand{\Cloc}{C_{\rm loc}}

\newcommand{\inv}{^{-1}}

\DeclareMathOperator{\sign}{sign}

\DeclareMathOperator{\id}{Id}
\newcommand{\abs}[1]{\mleft|#1\mright|}
\newcommand{\norm}[1]{\mleft\|#1\mright\|}

\renewcommand{\d}{\,d}
\newcommand{\dv}[2]{\frac{d #1}{d #2}}

\usepackage{mleftright}

\numberwithin{equation}{section}
\begin{document}

\title[KS-FKPP Traveling waves]{Traveling waves for the Keller-Segel-FKPP equation with strong chemotaxis}
%\title{Traveling waves for the Keller-Segel-FKPP equation with strong chemotaxis}
\author{Christopher Henderson}
\address[Christopher Henderson]{Department of Mathematics, University of Arizona}
\email{ckhenderson@math.arizona.edu}

\author{Maximilian Rezek}
\address[Maximilian Rezek]{Department of Mathematics, University of Arizona}
\email{maximilianrezek@math.arizona.edu}

\begin{abstract}
	We show that there exist traveling wave solutions of the Keller-Segel-FKPP equation, which models a diffusing and logistically growing population subject to chemotaxis. In contrast to previous results, our result is in the strong aggregation regime; that is, we make no smallness assumption on the parameters. The lack of a smallness condition makes $L^\infty$-estimates difficult to obtain as the comparison principle no longer gives them ``for free." Instead, our proof is based on suitable energy estimates in a carefully tailored uniformly local $L^p$-space.  Interestingly, our uniformly local space involves a scaling parameter, the choice of which is a crux of the argument. Numerical experiments exploring the stability, qualitative properties, and speeds of these waves are presented as well.
\end{abstract}

\maketitle

\section{Introduction}\label{sec.intro}

%In microbiology, chemotaxis refers to the migration of cells toward chemical attractants or away from chemical repellents \cite{schaechter2009encyclopedia}. For example, the movement of bacteria toward glucose (an attractant) and away from alcohol (a repellent) is a chemotaxis system. One of the first notable mathematical models of chemotaxis was presented by Keller and Segel in \cite{keller1970initiation, keller1971model}. Their model was specifically concerned with describing the concentration of the slime mold \textit{Dictyostelium discoideum} and the concentration of the chemical attractant cyclic adenosine monophosphate (cAMP).

The main goal of this work is to analyze front propagation phenomena in an FKPP-Keller-Segel system:
\begin{equation}\label{eq.model}
    \begin{cases}
    u_t+\chi(uv_x)_x=u_{xx}+u(1-u) & \text{in }(0,\infty)\times\R,\\
    -dv_{xx}=u-v & \text{in }(0,\infty)\times\R,
    \end{cases}
\end{equation}
where $d>0$ and $\chi\in\R$.  We point out that the unique bounded solution of the second equation is given in terms of a convolution:
\begin{equation}\label{eq.vConv}
    v=K_d*u
    \qquad\text{where}\qquad K_d(x)=\frac{1}{2\sqrt{d}}e^{-\frac{\abs{x}}{\sqrt{d}}}.
\end{equation}
The interpretation of equation~\eqref{eq.model} is the following: $u$ represents the population density of a species that diffuses, reproduces and competes logistically, as well as interacts intraspecifically via a ``chemical signal'' $v$.  In this model, that interaction is chemotaxis: each individual both secretes a chemical signal and moves in response to the chemical signal of its ``neighbors.''  
The sensitivity constant, $\chi\in\R$, describes each individual's perception and response to the chemical signal, with its magnitude encoding the strength of the chemotaxis, and the sign of $\chi$ determines if the chemotaxis is aggregative $(\chi>0)$ or dispersive $(\chi<0)$. %The growth source, $u(1-u)$, controls the organism's reproduction/death rate. 
Roughly, the diffusion coefficient, $d>0$, is the length-scale on which chemotaxis acts. % regulates the organism's propensity to congregate in more highly concentrated regions.
Chemotaxis is a well-studied phenomenon that is often seen in slime molds and bacteria~\cite{keller1970initiation,Murray,Perthame_transport}.

We are interested in traveling wave solutions to \eqref{eq.model} and their speed (we define these terms in \Cref{def.TWS}).  Traveling waves have previously been studied for similar reaction-diffusion equations including the Fisher-KPP equation, which arises when $\chi = 0$ in \eqref{eq.model}.  In this case, the speed of the slowest traveling wave is $2$, and this minimal speed wave describes the long-time behavior of solutions with ``localized'' initial data.  This is typical of reaction-diffusion models: the (minimal speed) traveling waves are stable in a suitable sense, meaning that their speed and profile generically describes the invasion of a species into a new environment.  We are interested in establishing a more complete understanding of the effects chemotaxis has on traveling waves.

% relates to the spreading of a species subject to chemotaxis, which is an intraspecific aggregation ($\chi > 0$) or dispersion ($\chi < 0$) effect.    Roughly, $u$ represents the population density of the species, while $v$ represents the density of the chemical signal. 
% %This model can be applied to a number of different contexts in the physical and social sciences, but usually the unknown $u$ represents the population density of a species experiencing an intraspecific aggregation or dispersion effect such as chemotaxis.
% The sensitivity constant, $\chi\in\R$, describes the organism's perception and response to the chemical stimulus. The magnitude of $\chi$ encodes the strength of the chemotaxis, and the sign of $\chi$ determines if the chemotaxis is aggregative $(\chi>0)$ or dispersive $(\chi<0)$. The growth source, $u(1-u)$, controls the organism's reproduction/death rate. The diffusion coefficient, $d>0$, regulates the organism's propensity to congregate in more highly concentrated regions.

The spreading properties of~\eqref{eq.model} have been the subject of great interest recently.  At the level of traveling waves, their behavior is fairly well understood as long as $\chi$ is not too positive.  Indeed, traveling waves have been shown to exist under the condition
\begin{equation}\label{e.c040202}
	\chi<\min\{1,d\}.
\end{equation}
This is contained in~\cite{NadinPerthameRyzhik} when $\chi > 0$ and~\cite{GrietteHendersonTuranova} when $\chi < 0$ (see also~\cite{SalakoShen2,SalakoShen4}).  We explain the technical importance of~\eqref{e.c040202} in \Cref{s.discussion}.  In this regime, it is known that, roughly, if $|\chi|$ and $d$ are not ``too big,'' any minimal speed traveling wave has speed $2$, while the minimal speed tends to infinity if $\chi \to -\infty$~\cite{NadinPerthameRyzhik,GrietteHendersonTuranova, henderson2021slow, SalakoShenXue}.  One expects, and sees numerically~\cite{AveryHolzerScheel}, a transition at a critical curve in the $(\chi,d)$ parameter space.  The Cauchy problem is more difficult, but certain aspects of spreading have been established~\cite{SalakoShen1,SalakoShen2, HamelHenderson}.  We note that there is an enormous PDE literature on chemotaxis focused on aspects like blow-up~\cite{DolbeaultPerthame, BlanchetDolbeaultPerthame, CarlenFigalli, MR2412327, MR3001763, MR1709861,MR2445771},  as well as front propagation-like questions for the many related models~\cite{Bramburger, SalakoShen3, fu2021sharp,GrietteMagalZhao,FuGrietteMagal-JMB,FuGrietteMagal-DCDSB,Bertsch-etal-2020,PerthameVauchelet,KimPozar, KimTuranova,PerthameQuirosVazquez, MR4401504,MR4219122,MR4303774, MR1709861, Calvez}.  The literature is truly vast, and this is only a small sampling of it.

Our interest in this paper is to investigate what happens when condition~\eqref{e.c040202} is not satisfied. Specifically, we construct traveling waves for any $\chi,d>0$, which requires a new approach to {\em a priori} estimates for the traveling wave problem associated to~\eqref{eq.model}.  We complement this with a numerical investigation of the behavior of the Cauchy problem that reveals (1) the minimal speed appears to always be $2$ in contrast to what happens when $\chi \to-\infty$, and (2) the possible existence of a bifurcation: when $\chi$ is large relative to $d$, pulsating fronts (roughly, traveling waves with patterns in the back) appear.

\subsection{Main result}

Let us begin by defining the notion of a traveling wave in our context. 

\begin{definition}\label{def.TWS}
A \textbf{traveling wave solution} to \eqref{eq.model} is a triple $(c,U,V)$ such that
\begin{enumerate}[topsep=1ex, itemsep=1ex,]%[topsep=0pt,itemsep=-1ex,partopsep=1ex,parsep=1ex]
    \item $0<U\in C^2(\R)\cap L^\infty(\R)$,
    \item $c>0$,
    \item $u(t,x)=U(x-ct)$ and $v(t,x)=V(x-ct)$ solve \eqref{eq.model}, and
    \item $\displaystyle{\liminf_{x\to-\infty}U(x)>0}$, $\displaystyle{\lim_{x\to\infty}U(x)=0}$.
%    \item $\displaystyle{\lim_{x\to-\infty}V(x)=1}$, and $\displaystyle{\lim_{x\to\infty}V(x)=0}$.
\end{enumerate}
We refer to $c$ and $U$ as the \textbf{speed} and \textbf{profile}, respectively.
\end{definition}

\noindent  Applying the coordinate change in \Cref{def.TWS}, the existence of a traveling wave solution to~\eqref{eq.model} is equivalent to finding a triple $(c,U,V)$ that solves the system:
\begin{equation}\label{eq.TWSmodel}
    \begin{cases}
    -cU'+\chi(UV')'=U''+U(1-U) & \text{in }\R,\\
    -dV''=U-V & \text{in }\R.
    \end{cases}
\end{equation}

We now state our main theorem.
\begin{theorem}\label{t.thm}
    For all $\chi,d>0$, there exists a traveling wave solution to \eqref{eq.model}.  Moreover, the speed $c$ of any traveling wave solution satisfies $c\geq 2$.
\end{theorem}

We make a few comments about the theorem.  First, note that this result actually holds for \textit{any} $\chi\in\R$.  As mentioned above, the case when $\chi$ is nonpositive has been considered in \cite{henderson2021slow, GrietteHendersonTuranova}, and the case when $\chi$ is small relative to $d$ was considered in, e.g.,~\cite{NadinPerthameRyzhik, SalakoShenXue}; however, the smallness condition~\eqref{e.c040202} was crucial to the argument in these works.

Second, it is natural to ask if our choice of $K_d$ in~\eqref{eq.vConv} is necessary for our argument or if they would hold for a similar model but with nonlocal advection $v = K_d * u$ defined by a different $K_d$.  This type of model was considered in \cite{henderson2021slow}, where the interested reader can find natural conditions on $K_d$.  Our proof seems quite flexible in this regard.  The form of $K_d$ is mainly used in relating the $L^\infty$-norm of $V$ to a certain ``uniformly local'' norm of $U$ (see \Cref{d.UnifLocalLp} and \Cref{l.UnifLocBoundOnVa}).  If, e.g., $K_d\sim (1 + |x|)^{-1-\alpha}$ for $\alpha>0$, it seems one could replace $\phi$ in \Cref{d.UnifLocalLp} by $(1 + |\eul x|)^{-1-\alpha}$ and reproduce the same estimate. %\edit{Should we present or summarize the properties $K_d$ must satisfy? I found the discussion about $K_d$ on page 2 in \cite{henderson2021slow} useful for understanding the model's physical interpretation, but I suppose just citing \cite{henderson2021slow} accomplishes this and keeps the paper focused.}

Finally, let us discuss the role of ``minimal speed'' traveling waves.  We say that a traveling wave $(c^*,U^*,V^*)$ is a minimal speed traveling wave if $c^*\leq c$ for all traveling wave solutions $(c,U,V)$.  Often, for equations with a logistic reaction term such as~\eqref{eq.model}, there is a traveling wave for every speed $c\geq c_*$ but the minimal speed one is stable with regard to ``localized'' initial data~\cite{AronsonWeinberger}.  One might suspect that there is an infinite half-line of speeds here as well; however, we do not address this, nor do we address whether the wave we construct is the minimal speed one.  On the other hand, the procedure we implement ``should'' return the minimal speed wave, as it does when $\chi = 0$.  The stability of the wave that we construct is a subtle issue.  As we discuss in \Cref{s.numerics}, it seems solutions starting from Heaviside initial data form a speed $2$ traveling wave when $\chi \leq (1 + \sqrt d)^2$; however, it appears they form a ``pulsating front'' otherwise (see~\cite[Section~2.2]{XinBook} for a definition, as well as~\cite{SKT} for the first work introducing the concept).  We leave the discussion of this to the sequel.

\subsection{Discussion of the proof}\label{s.discussion}

One standard procedure for establishing the existence of traveling wave solutions for similar reaction-diffusion systems is to first consider the ``slab problem," i.e., \eqref{eq.TWSmodel} but on a finite interval $[-a,a]$ for $a\gg1$ with the boundary conditions $U_a(-a)=1$ and $U_a(a)=0$. One then obtains existence of solutions on sufficiently large slabs via the Leray-Schauder topological degree theory, the key step of which is to  establish suitable bounds on $c_a$ and $U_a$.  In order to take the limit as $a\to\infty$ to obtain the traveling wave solution, it is crucial that the established bounds depend only on $d$ and $\chi$ and not on $a$. %in $\R$ and $C^{2,\alpha}([-a,a])$ (for $\alpha\in(0,1)$), respectively, by a constant $C_0$ that depends solely on $d$ and $\chi$:
% \begin{equation}\label{eq.C_0bound}
%     \frac{1}{C_0}
%     \leq
%     c_a+\norm{U_a}_{L^\infty}
%     \leq
%     C_0.
% \end{equation}
%and to apply the Leray-Schauder topological degree theory to 

%This bound allows us to apply Leray-Schauder topological degree theory to establish the existence of solutions on sufficiently large slabs. And finally, we use the sliding plane method to extend a solution to the slab problem to $\R$.

\subsubsection*{Main difficulties}

For the particular model considered in this article, the difficult part of this approach is establishing the uniform bounds on $c_a$ and $U_a$. Previous works on traveling wave solutions of~\eqref{eq.model}, assume a smallness condition on $d$ and $\chi$ when $\chi>0$, whereas we do not. To illustrate why such a condition is helpful, we briefly outline the argument that establishes an $L^\infty$-bound on $U_a$.  Suppose $\sfrac{\chi}{d}<1$ and $U_a$ achieves its maximum $M$ at $x_M\in(-a,a)$. By combining the two equations in \eqref{eq.TWSmodel}, we have
\begin{equation}\label{10.03.2}
    -c_aU_a'+\chi U_a'V_a'
    =
    U_a''+U_a\mleft(1-U_a-\frac{\chi}{d}(U_a-V_a)\mright).
\end{equation}
Thus, at $x_M$, we have
\begin{equation}\label{e.c120201}
    \begin{split}
    0	&= -c_aU_a'+\chi U_a'V_a'
    =U_a''+M\mleft(1-M-\frac{\chi}{d}(M-V_a)\mright)
    \\&
    \leq 
    M\mleft(1-M+\frac{\chi}{d}(M-V_a)\mright)
    \leq M \mleft(1 - M +\frac{\chi}{d}M\mright),
    \end{split}
\end{equation}
which immediately yields the bound
\begin{equation}\label{10.03.1}
    \norm{U_a}_{L^\infty} = M \leq \frac{1}{1-\sfrac{\chi}{d}}.
\end{equation}
We note that a na\"ive phase plane argument using ``trapping region'' arguments leads to the same obstruction.
% This gives the following estimate on $M$: \edit{check}
% \begin{equation}\label{10.03.1}
%     M\leq\frac{1}{1-\frac{\chi}{d}}.
% \end{equation}

\subsubsection*{Tello-Winkler approach}

From \eqref{e.c120201}, we see that the only hope to make the above argument work without the smallness condition is to use the $V_a$ term that is dropped in~\eqref{e.c120201}.  This would require a lower bound on $V_a$; % of the form $M(1 + o(\sfrac{d}{\chi}))$;
however, any lower bound on $V_a$ will necessarily require regularity estimates of $U_a$, which, in turn depend on $c_a$ due to~\eqref{eq.TWSmodel}.  On the other hand, the standard approach for estimating the speed $c_a$ for the slab problem yields
\begin{equation}\label{e.c120901}
    c_a \leq 2 + \mleft(\frac{\chi}{\sqrt d} + \frac{\chi}{d}\mright) \|V_a\|_{L^\infty}
\end{equation}
(see \Cref{l.c_a<=|U_a|}). 
Alas, we find ourselves in a `loop' of inequalities, which is the main difficulty in proving \Cref{t.thm}.

%, without the smallness condition $\sfrac{\chi}{d}<1$,  we cannot rely on a comparison principle argument to bound $U_a$. Thus, other methods must be used to show boundedness. If we are to apply elliptic regularity theory to bound $U_a$, it follows from \eqref{10.03.2} that such a bound involves $c_a$. Likewise, a bound on $c_a$ involves $U_a$. Therefore, to close this `loop,' it is necessary to bound either $c_a$ or $U_a$.

Our approach to bounding $c_a$ and $U_a$ is based on the method employed in \cite{TelloWinkler} with some modifications. The authors in \cite{TelloWinkler} consider a slightly more general model than~\eqref{eq.model} on a bounded domain $\Omega\subseteq\R^n$ supplemented with Neumann boundary conditions for $u$ and $v$.  
%For clarity, let us show their argument in the context of~\eqref{eq.model}.   the following elliptic-parabolic initial-boundary value problem:
%\begin{equation}\label{e.telloModel}
%\begin{cases}
%   u_t=\Delta u-\chi\nabla\cdot(u\nabla v)+u(1-u) & \text{in } (0,T)\times\Omega,\\
%   -d\Delta v+v=u & \text{in } (0,T)\times\Omega,
%\end{cases}
%\end{equation}
%where $\Omega$ is a bounded domain in $\R^n$ with smooth boundary. 

We outline their argument used to obtain a uniform bound on $\norm{u(t,\cdot)}_{L^\infty(\Omega)}$ now.  They first show $\norm{u(t,\cdot)}_{L^p(\Omega)}$ is uniformly bounded for values of $p>1$ near $1$.  This is done by multiplying~\eqref{eq.model} by $u^{p-1}$, leveraging the negative quadratic term in $u(1-u)$, and carefully integrating by parts to obtain
\begin{equation}
    \frac{1}{p}\dv{}{t}\int_\Omega u^p \d x +(p-1)\int_\Omega u^{p-2}\abs{\nabla u}^2 \d x
    % &=
    % \frac{\chi}{d}\frac{(p-1)}{p}\mleft(-\int_\Omega u^p v+\int_\Omega u^2\mright)+\int_\Omega u^p(1-u)\nonumber\\
    \leq
    \mleft(\frac{\chi}{d}\frac{(p-1)}{p}-1\mright)\int_\Omega u^{p+1}\d x+\int_\Omega u^{p-1}\d x.\label{10.13.1}
\end{equation}
The coefficient of the first integral on the right side of \eqref{10.13.1} is made negative by choosing $p$ sufficiently close to $1$. 
As a result, the first term on the right controls the second term via H\"older's inequality:
\begin{equation}\label{e.102801}
	\int_\Omega u^{p-1}\d x
		\leq |\Omega|^\frac{1}{p} \mleft( \int_\Omega u^p \d x\mright)^\frac{p-1}{p}
	\quad\text{and}\quad
	\int_\Omega u^{p+1}\d x
		\geq \abs{\Omega}^{-p} \mleft( \int_\Omega u^p\d x \mright)^\frac{p+1}{p}.
\end{equation}
From here, one can use a dynamical argument to rule out $\|u\|_{L^p}$ of ever becoming ``too large.''  As is clear from~\eqref{e.102801}, the size of $\Omega$ affects the ultimate bound on $\|u(t,\cdot)\|_{L^p}$.  This is natural as we expect $u$ to equilibrate to an $O(1)$ steady solution, whence $\|u(t,\cdot)\|_{L^p} = O(|\Omega|^{1/p})$.

Afterwards, one can bootstrap the $L^p$-bound to higher $L^r$-estimates of $u$ by using the gradient term in~\eqref{10.13.1}.  The full regularity of $u$ follows by standard parabolic regularity arguments once the high enough integrability of $u$ is obtained.

\subsubsection*{Outline of the proof of \texorpdfstring{\Cref{t.thm}}{Theorem \ref{t.thm}}}

Two issues arise when applying this argument in our context: we require a bound that is independent of $a$ (the size of the domain), and we have no time dependence (recall we have made the traveling wave change of variables that turns the $u_t$ term into $-c U'$).  We discuss the former issue at greater length, and simply mention that the complication of the latter is to induce additional $c$ dependence in each estimate.

%The estimates in \eqref{e.102801} rely on the finiteness of $\Omega$, so this crucial step cannot be carried over to our context as the domain in \eqref{eq.TWSmodel} is $\R$. 
To overcome the first issue, we introduce a new variant of \textit{uniformly local $L^p$-spaces}.  Although we only use this with $p=2$, we state the general definition. %First introduced by Kato in \cite{kato1975cauchy} for studying hyperbolic systems, these spaces have more recently been used in the study of kinetic equations.

\begin{definition}\label{d.UnifLocalLp}
Fix $p \in [1,\infty)$, $\eul>0$, and $\psi\in\{C_c^\infty(\R):\norm{\psi}_{L^1}=1\text{ and }0\leq\psi\leq1\}$.  Let $\phi=e^{-\eul\abs{x}}*\psi$, and, for any $s\in \R$, let $\phi_s(x) = \phi(x-s)$.  We define the \textbf{uniformly local $\boldsymbol{L^p}$-norm} of a measurable function $f$ as
\[
    \norm{f}_{\lul^p}
    =
    \sup_{s\in\R}\|\phi_s^{1/p} f\|_{L^p}
    = \sup_{s\in\R}\mleft(\int_{-\infty}^\infty \phi_s(x)\abs{f(x)}^p \d x \mright)^\frac{1}{p}.
\]
\end{definition}
\noindent  The original uniformly local $L^p$-spaces were introduced by Kato~\cite{kato1975cauchy} to construct solutions to a hyperbolic system of equations and involves a compactly supported $\phi$.  These spaces have been used in various contexts, e.g., the Boltzmann equation~\cite{AMUXY}, in the last half century.

As we describe below, the uniformly local $\lul^2$-space presents many advantages for us. For one, it localizes our estimates to a domain of ``size'' $O(\sfrac1\eul)$, which, to a degree, solves the problem of not working on a finite domain.  %Any estimates obtained analogously as in~\eqref{10.13.1}-\eqref{e.102801} will depend on $\eul$, though. %Notice, however, that the Tello-Winkler estimates, were they possible to put our estimates depend on $\eul$ for  analogously to~\eqref{10.13.1}-\eqref{e.102801}.
An important feature of the uniformly local norm is that, when $\eul < \sfrac{1}{\sqrt{d}}$, we have
\begin{equation}\label{e.c040401}
    \|V_a\|_{L^\infty}
    \leq 
    C\norm{U_a}_{\lul^2},
\end{equation}
which is a consequence of the fact that $K_d\leq C\phi$. (See \Cref{l.varphiBounds} and \Cref{l.UnifLocBoundOnVa}.)

%\[
%    K_d 
%    \leq
%    C \phi
%\]
%(See \Cref{l.varphiBounds} and \Cref{l.UnifLocBoundOnVa}.)

With the observations above and some care, one can prove an energy estimate in this uniformly local space of the form (\Cref{l.energy}):
\begin{equation}\label{e.c120902}
	\|U_a\|_{\lul^2}^2
		\leq C \mleft( \frac{1}{\eul} + \eul c_a^2\mright).
%    \|U\|_{\lul^{p+1}}^{p+1}
%        + \| (U^{p/2})'\|_{\lul^2}^{2}
%        \leq C_{d,\chi}\mleft( \frac{1}{\eul} + \eul^p c^{p+1}\mright).
\end{equation}
Actually, boundary condition at $x=-a$ causes some issues, so what we really prove is that, for each $s$,
\begin{equation}\label{e.c040403}
	\int_{x_1}^a \phi_s U_a^2
		\leq C \mleft( \frac{1}{\eul} + \eul c_a^2 + \sqrt{\eul} \|U_a\|_{\lul^2}^2\mright),
\end{equation}
whenever $\|U_a\|_{L^\infty} > 2$.  Here, $x_1$ is the leftmost element of the $\sfrac32$ level set of $U_a$
%any fixed level set of $U_a$ between $(1,2)$ and below $\|U_a\|_\infty$.
%\[
%	x_1 := \min\{x>-a : U_a(x) = \ell\}$ for $\ell\in (1,2)$ that is less that $\|U_a\|_{L^\infty}$.  
This ensures that $U_a'(x_1) \geq 0$ and $U_a'(a) \leq 0$, which are ``good'' signs for all of the boundary terms that arise.  %because they have a ``good'' sign.  
Additionally, notice that the term on the left in~\eqref{e.c040403} is comparable to the $\lul^2$-norm of $U_a$ up to $\eul$ factors.  This, and the ``small'' $\sqrt{\eul}$ factor in front of the $\lul^2$-term, allow us to deduce~\eqref{e.c120902}.

Notice that we have not decoupled the dependence of the norms of $U_a$ on $c_a$, yet; however, we have ``won'' a small parameter of $\eul$ in front of the ``bad'' term $c_a^2$.  Indeed, standard arguments show that
\[
	c_a
	\leq
	2+\chi\|V_a''\|_{L^\infty}+\chi\|V_a'\|_{L^\infty},
\]
and it is straightforward to check that $d \|V_a''\|_{L^\infty}, \sqrt{d}\|V_a\|_{L^\infty} \leq \|V_a\|_{L^\infty}$ (see \Cref{l.c_a<=|U_a|}).  Combining this with~\eqref{e.c120902} and~\eqref{e.c040401}, we find
\[
	c_a^2 \leq C \mleft( \frac{1}{\eul} + \eul c_a^2\mright),
\]
at which point we deduce the desired bound on $c_a$ by further decreasing $\eul$.  An $\lul^2$-bound on $U_a$ follows then from this bound on $c_a$ and~\eqref{e.c120902}.  At that point, we now have bounds on the coefficients of~\eqref{10.03.2} (recall~\eqref{e.c040401} for the bound on $V_a$ and \Cref{l.VaInftyBounds} for the bound on $V_a'$), so a standard argument to upgrade a local $L^2$-estimate to an $L^\infty$-bound may be applied.

%
%
%
%combining this with~\eqref{e.c120901} and a suitable Morrey-type inequality (\Cref{l.morrey}), we obtain
%\[
%   % \|U\|_{\lul^{p+1}}^{p+1}
%    %    + \| (U^{p/2})'\|_{\lul^2}^{2}
%    \|U\|_{L^\infty}^{p+1}
%        \leq C_{d,\chi}\mleft( \frac{1}{\eul} + \eul^p \|U\|_{L^\infty}^{p+1}\mright),
%\]
%which yields a bound on $\|U\|_{L^\infty}$ (and, thus, on $c$ as well) after choosing $\eul$ small enough.  This last step is a novelty in our argument: building a smallness parameter in the definition of the uniformly local space.

We point out that the main estimate~\eqref{e.c120902} is, in fact, quite different from the estimate used in~\cite{TelloWinkler}, although inspired by it.

\subsection*{Notation}%\currentpdfbookmark{Notation}{Notation}
Unless needed for clarity, we drop the integrand's dependence on the variable of integration, e.g., we write
\[
    \int_0^1 U
    \quad\text{instead of}\quad
    \int_0^1 U(x)\d x.
\]
Also, we suppress the dependence on the domain for function spaces' norms if the domain is $\R$. For example, we write $\norm{U}_{L^p}$ instead of $\norm{U}_{L^p(\R)}$. Lastly, $C$ represents a positive constant that may change line-by-line and depends only on $d$ and $\chi$.  In general, it is assumed that all constants and conditions depend on $d$ and $\chi$.% does not depend on $d$, $\tau$, or $\chi$. %If $C$ depends on one or more of these parameters, a subscript is used to denote this dependence. For example, $C_{d,\chi}$ represents a positive constant that depends on $d$ and $\chi$, but not $\tau$.

\section{Numerical simulations}\label{s.numerics}
%\subpdfbookmark{Numerical Simulations}{NumericalSimulations}

The Cauchy problem of~\eqref{eq.model} is beyond the scope of this paper.  We performed numerical experiments to investigate this question.  
%Numerical simulations of our model reveal that the qualitative properties of traveling waves with an attractive chemotactic effect ($\chi>0$) differ greatly from when the effect is repulsive ($\chi<0$). 
We essentially use the ``upwind" numerical scheme from \cite{fu2021sharp}. %, modifying it slightly to account for our advection term.  %\edit{How much of a modification to Quentin’s upwind scheme would we need to make to warrant mentioning it? Unless I am missing something, our upwind scheme is the same as Quentin’s; the only modifications I don’t think are worth mentioning.}.   
Numerical simulations inherently take place on a finite domain, so we impose Neumann boundary conditions and take care to stop the simulation long before the front approaches the boundary. %When the effect is repulsive, the shape of the traveling wave does not change with time and resembles a logistic curve (see \Cref{fig:chi=-1}). 
%The speeds of the traveling waves are calculated by tracking a fixed level set.  

Often (minimal speed) traveling waves are unique and stable, which means that solutions $u$ to~\eqref{eq.model} converge in a $ct + o(t)$ moving frame to the minimal speed traveling wave solution.  Further, in other simpler models, minimal speed traveling waves can be constructed by the process performed in this paper.  This suggests that one might expect the solution we construct here to be the minimal speed traveling wave and that it is the limit of solutions to the Cauchy problem.  We see, in the sequence, that this is not the case when $\chi$ is large.

The behavior of $u$ changes greatly as one changes $\chi$.  It was already known that when $\chi$ is sufficiently negative, traveling waves are ``sped up'', and that when $\chi$ is close enough to zero (with no constraint on the sign), the wave speed is the same as when $\chi = 0$, i.e., $c=2$~\cite{henderson2021slow, GrietteHendersonTuranova}.  What we find here is that the wave speed seems to not change as $\chi$ is increased.  On the other hand, a Hopf-type bifurcation occurs: the traveling wave seems to destabilize and a pulsating front appears and seems to be stable.

\begin{figure}[t]
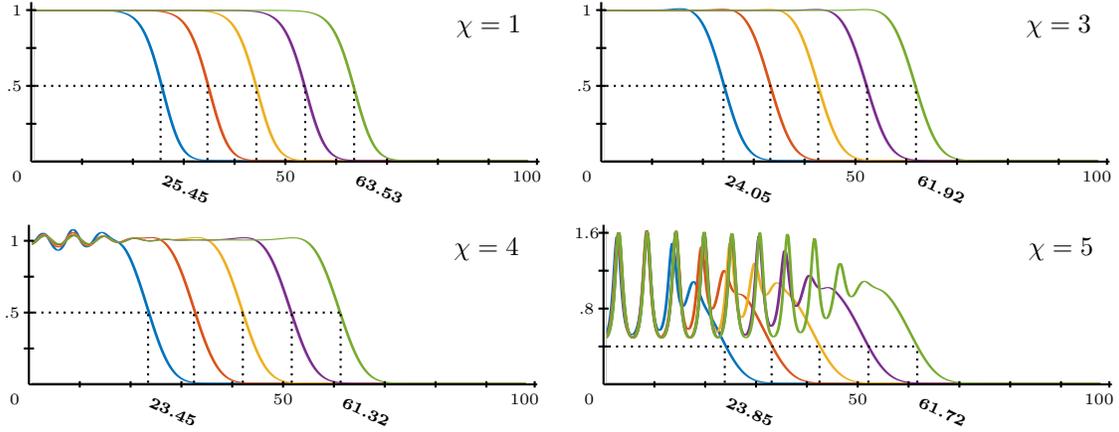

\centering
\speedplots
\caption{In all images above, $d=1$ with initial data $u_{\rm in} = \exp\{-2(x-10)_+^2/5\}$.  We took spatial step size $\Delta x = 0.2$ and time step size $\Delta t = (\Delta x)^2/10$.  In the first three plots above, we can approximately compute the speed by tracking the $\sfrac12$ level set, and in the last one, we use $\sfrac25$ level set (this avoids issues with the pattern making the $\sfrac12$ level set have multiple elements).  In a fixed plot, each curve is the profile $u$ at five units of time beyond the profile to its left; indeed, the curves were sampled at $t=10,15,20,25,30$.}
\label{f.speedplots}
\end{figure}

Let us look more closely at the speed first.  Computing the approximate speed by subtracting location of the leftmost level set from the rightmost and dividing by 20 (4 intervals of 5 units of time), we find:
\[
	\begin{aligned}
		&c_{\chi = 1} \approx \frac{63.53 - 25.45}{20}
			\approx 1.90,
		\qquad
		&c_{\chi = 3} \approx \frac{61.92 - 24.05}{20}
			\approx 1.89,
		\\&
		c_{\chi = 4} \approx \frac{61.32 - 23.45}{20}
			\approx 1.89,
		\qquad\text{ and } \qquad
		&c_{\chi = 5} \approx \frac{61.72 - 23.85}{20}
			\approx 1.89.
	\end{aligned}
\]
We note that, despite the clear change in qualitative behavior in \Cref{f.speedplots} as $\chi$ increases, the speed remains essentially constant.  Additionally, given the ``large'' $\Delta x$ and $\Delta t$ used in our simulations and the difficulties of computing the speed of waves numerically (see, e.g., the discussion in the introduction of~\cite{DemircigilFabreges}), it is seems heuristically clear that $c = 2$ is within the margin of error.  We do not investigate this more systematically; however, it is tempting to conjecture that $c=2$ for $\chi >0$.

\begin{figure}
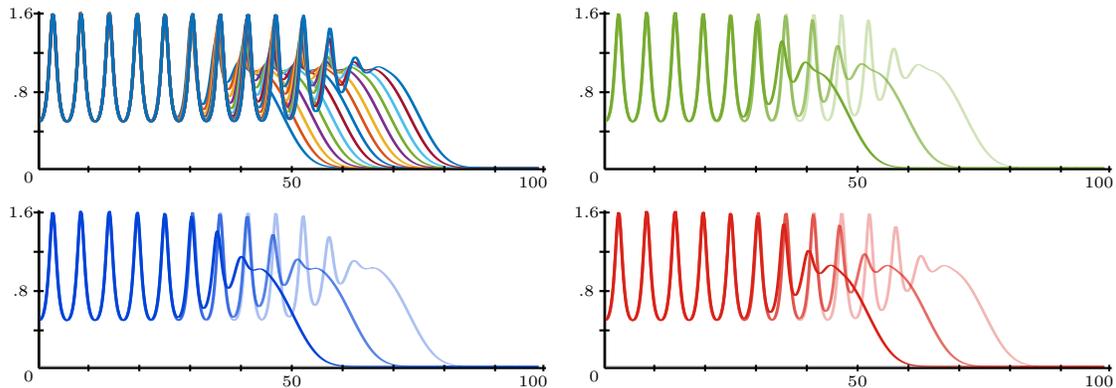

\begin{center}
\patternplots
\caption{The above images show plots for $\chi=5$ that reveal the time-periodic nature of the profile in the moving frame.  The top left plot is of $u(t,\cdot)$ for all integer $t$ between $24$ and $38$.  It appears that the profile is $3$ periodic in time.  To show this, the top right plot is of $t=24, 30, 36$, the bottom left plot is of $t=25,31,37$, and the bottom right plot is of $t=26,32,38$.  Notice that, when the time difference is a multiple of $3$ the profile appears to be an exact copy but shifted to the right.   (We took time differences of $6$ in order to minimize overcrowding in the plots, but the same features are present when the time difference is $3$.)}\label{fig2}
\end{center}
\end{figure}

We now turn our attention to the change in behavior that occurs when
\begin{equation}
	\chi = 4 = (1 + \sqrt d)^2 %\mleft(1 + \sqrt d\mright)^2
\end{equation}
(recall that $d=1$ in our simulations).  In \Cref{f.speedplots}, we see periodic behavior begin to appear here.  For $\chi = 4$, the ``wiggles'' dissipate in time; however, they are permanent in the $\chi = 5$ plot.  In a sense, this is not unexpected in that a linear stability analysis of the constant state $u \equiv 1$ reveals its stability if and only if $\chi < (1 + \sqrt d)^2$~\cite{NadinPerthameRyzhik}.  
We note further work done on constructing patterns in similar models~\cite{DucrotFuMagal, FiedlerPolacik}.  
Numerically, it appears that $u$ converges (in a moving frame) to a {\em pulsating front} (see \Cref{fig2}); that is, a solution to~\eqref{eq.model} of the form $u(t,x) = U(x, x-ct)$, where $U$ is periodic in the second variable~\cite{SKT,XinBook}.  To our knowledge, it is an open question whether pulsating fronts exist for~\eqref{eq.model}.  Our numerics, however, suggest that these should exist and be stable when
\begin{equation}\label{e.c041101}
	\chi > (1 + \sqrt d)^2.
\end{equation}
To our knowledge, the existence of pulsating fronts that connect $0$ to a nonconstant (periodic) steady states have not been proven for any reaction-diffusion models.

Let us note the gap in the parameters $(\chi,d)$ between where previous works had constructed traveling waves (under the assumption~\eqref{e.c040202}) and where they are expected to no longer be stable (when~\eqref{e.c041101} holds).  Indeed, the interval
\[
	\min\{1,d\}
		\leq \chi
		\leq (1 + \sqrt d)^2
\]
can be arbitrarily large, depending on $d$.  It is in this regime that we expect our result \Cref{t.thm} to be relevant; however, proving the stability of the wave constructed here is beyond the scope of the current paper.

In view of recent results on finite domains~\cite{PainterHillen,WangWinklerXiang,WinklerJNS}, one might expect that the ``peaks'' behind the front in the pulsating front will grow to infinity as $\chi$ tends to infinity.  Interestingly, we did not observe this in our numerical simulations. Further study is necessary to determine if this is a limitation of the simulations or if it reflects the true behavior of the model.

\section{The problem on a finite slab}\label{sec.slab}

To establish the existence of a solution to \eqref{eq.TWSmodel}, we first show there exists a solution on every finite ``slab," i.e., on every finite interval $(-a,a)$, where $a>0$.  In order to guarantee the positivity of $U_a$, we also slightly adjust the equation.  To this end, consider \eqref{eq.TWSmodel} on $(-a,a)$ with added boundary conditions:
\begin{equation}\label{eq.SlabModel}
    \begin{cases}
    -c_aU_a'+\tau\chi(U_a\va')'
    =
    U_a''+(U_a)_+(1-U_a) 
    & 
    \text{in }(-a,a),\\
    U_a(-a)=1, 
    \quad
    U_a(a)=0,
    \quad
    \max_{x\geq0}\widetilde{U}_a(x)=\theta>0,
    \end{cases}
\end{equation}
where $x_+ := \max\{0,x\}$, $\theta\in(0,\sfrac14)$ is a small fixed parameter, $\tau\in[0,1]$ is a parameter used when we apply Leray-Schauder degree theory in \Cref{p.slab_existence}, and
\begin{equation}\label{eq.vaConv}
    \va=\widetilde{U}_a*K_d
    \qquad
    \text{where}
    \qquad
    \widetilde{U}_a(x)
    =
    \begin{cases}
    1 & \text{if } x\leq-a,\\
    U_a(x) & \text{if } x\in(-a,a),\\
    0 & \text{if } x\geq a.
    \end{cases}
\end{equation}
A direct computation with this definition of $\widetilde{U}_a$ and $\va$ yields
\begin{equation}\label{eq.va''}
    -d\va''=\widetilde{U}_a-\va \qquad \text{in }\R.
\end{equation}
Note that, on the slab, we have
\begin{equation}\label{eq.va''slab}
    -d\va''=U_a-\va \qquad \text{in }[-a,a].
\end{equation}
The subscripts indicate the solution's dependence on $a$. Note also that a solution $(c_a,U_a,\va)$ to \eqref{eq.SlabModel} also depends on $\theta$ and $\tau$, but, unless needed for clarity, we suppress this dependence in our notation.

Let us make two additional technical remarks.  First, note that $\|\widetilde{U}_a\|_{L^\infty}=\norm{U_a}_{L^\infty([-a,a])}$. Therefore, for notational convenience, we use $\norm{U_a}_{L^\infty}$ to denote both $\|\widetilde{U}_a\|_{L^\infty}$ and $\norm{U_a}_{L^\infty([-a,a])}$ whenever either term arises.

Second, note that first equation in~\eqref{eq.SlabModel} differs slightly from that of~\eqref{eq.model} because one of the $U_a$ terms is changed to $(U_a)_+$.  This allows us to immediately prove that $U_a>0$ (see \Cref{p.U_a>=0}), at which point we see that the first equation in~\eqref{eq.SlabModel} and that of~\eqref{eq.model} agree.

\subsection{Preliminaries}

In this section, we collect a few useful inequalities that are deployed in the main argument.  In particular, we establish the positivity of solutions to \eqref{eq.SlabModel}, as well as bounds on $\va$ (and its first two derivatives).  The results in this subsection are either proved using standard arguments or are straightforward to deduce.

\subsubsection{Estimates on \texorpdfstring{$U_a$}{Uₐ} and \texorpdfstring{$\va$}{Vₐ}}%\subpdfbookmark{Estimates on Uₐ and Vₐ}{EstimatesOnU_aAndV_a}

First we point out some easy estimates on $V_a$ coming from $U_a$.

\begin{lemma}[Bounds on $\va$, $\va'$, and $\va''$]\label{l.VaInftyBounds}
If $(c_a,U_a,\va)$ solves \eqref{eq.SlabModel} and $U_a\geq 0$, then, for all $x\in \R$,
%\begin{equation}\label{eq.VaInftyBounds}
%    \norm{\va}_{L^\infty}\leq\norm{U_a}_{L^\infty},
%    \qquad
%    \norm{\va'}_{L^\infty}\leq\frac{1}{\sqrt{d}}\norm{U_a}_{L^\infty},
%    \qquad\text{and}\qquad
%    \norm{\va''}_{L^\infty}\leq\frac{1}{d}\norm{U}_{L^\infty}.
%\end{equation}
%Also, if $\widetilde{U}_a$ is nonnegative, then, for all $x\in\R$,
\begin{equation}\label{eq.PtWiseVaBound}
    |\va'(x)|
    \leq
    \frac{1}{\sqrt{d}}\va(x).
\end{equation}
\end{lemma}
\begin{proof}
%The estimates in \Cref \eqref{eq.VaInftyBounds} follow from Young's convolution inequality, so we omit them.
%:
%\begin{align*}
%    \norm{\va}_{L^\infty}
%    &=\|\widetilde{U}_a*K_d\|_{L^\infty}
%    \leq\norm{U_a}_{L^\infty}\norm{K_d}_{L^1}
%    =\norm{U_a}_{L^\infty},\\
%    \norm{\va'}_{L^\infty}
%    &=\|\widetilde{U}_a*K_d'\|_{L^\infty}
%    \leq\norm{U_a}_{L^\infty}\frac{1}{\sqrt{d}}\int_{-\infty}^\infty\abs{-\sign(x)K_d(x)}\d x
%    =\frac{1}{\sqrt{d}}\norm{U_a}_{L^\infty}, \text{ and}\\
%    \norm{\va''}_{L^\infty}
%    &=\|\widetilde{U}_a*K_d''\|_{L^\infty}
%    \leq\norm{U_a}_{L^\infty}\frac{1}{d}\int_{-\infty}^\infty\abs{\sign(x)^2K_d(x)-\delta(x)}\d x
%    =\frac{1}{d}\norm{U_a}_{L^\infty}.
%\end{align*}
This follows from a simple computation using 
\[
    \abs{\va'(x)}
    =
    |(K_d'*\widetilde{U}_a)(x)|
    =
    \frac{1}{\sqrt{d}}\abs{\int_{-\infty}^\infty -\sign(y)K_d(y)\widetilde{U}_a(x-y)\d y}
    \leq
    \frac{1}{\sqrt{d}}\va(x).\qedhere
\]
\end{proof}

Next, we check that $U_a$ is, in fact, nonnegative.

\begin{lemma}[Nonnegativity of $U_a$ and $\va$]\label{p.U_a>=0}
If $(c_a,U_a,\va)$ solves \eqref{eq.SlabModel}, then both $U_a$ and $\va$ are nonnegative on $[-a,a]$ and positive in $(-a,a)$.
\end{lemma}
\begin{proof}
Note that the result follows by the strong maximum principle once we show that
\begin{equation}\label{e.c032101}
	\min_{[-a,a]} U_a \geq 0.
\end{equation}
We establish~\eqref{e.c032101} by contradiction.  Suppose that there exists $x_0$ such that
\begin{equation}\label{e.c032501}
	U_a(x_0) = \min_{[-a,a]} U_a < 0.
\end{equation}
Then, due to the boundary conditions~\eqref{eq.SlabModel}, $x_0 \in (-a,a)$.  We see from~\eqref{eq.SlabModel} and~\eqref{eq.va''slab} that, at $x_0$,
%\begin{equation}
%    0 \leq U_a'' + 0
%		= U_a'' + (U_a)_+ (1- U_a)
%		= - c_a U_a' + \tau \chi (U_a' V_a' + U_a V_a'')
%		= \tau \chi U_a V_a''
%		= \tau \chi U_a (V_a - U_a).
%\end{equation}
\begin{equation}\label{e.c032102}
	\begin{split}
	0 &\leq U_a'' + 0
		= U_a'' + (U_a)_+ (1- U_a)
		= - c_a U_a' + \tau \chi (U_a' V_a' + U_a V_a'')
		\\&
		= \tau \chi U_a V_a''
		= \tau \chi U_a (V_a - U_a).
	\end{split}
\end{equation}
We note that $U_a(x_0) < 0$. 
Additionally, using~\eqref{eq.vaConv},~\eqref{e.c032501}, and the fact that $\widetilde U_a$ is nonconstant, we find
\begin{equation}
	V_a(x_0)
		= K_d* \widetilde U_a(x_0)
		> \min \widetilde U_a = U_a(x_0).
\end{equation}
It follows that, at $x_0$,
\begin{equation}
	U_a (V_a - U_a) < 0,
\end{equation}
which contradicts~\eqref{e.c032102} and concludes the proof.
\end{proof}

\subsubsection{A lower bound on the speed}

We must ensure that $c_a$ remains uniformly positive.  The first step to doing that is to show that $V_a$ and $U_a$ are not ``too big'' far on the right.  

%The following lemma establishes that $\va(x)$ and $\abs{\va'(x)}$ are bounded by $C_d\theta$ for sufficiently large $x>0$.

\begin{lemma}\label{l.v v'<theta}
There exits $L>0$, such that, for all $x\geq L$,
\begin{equation}\label{eq.v v'<theta}
    \va(x)\leq2\theta
    \qquad
    \text{and}
    \qquad
    \abs{\va'(x)}\leq\frac{2}{\sqrt{d}}\theta.
\end{equation}
The constant $L$ is independent of $a$ but does depend on $\|U_a\|_{L^\infty}$ and $\theta$.
\end{lemma}
\begin{proof}
%Recall that $K_d$ is even~\eqref{eq.vConv}.  Hence, a change of variables allows us to write, for all $x\in\R$,
%\[
%    \va(x)
%    =
%    \int_0^\infty K_d(y)\mleft[\widetilde{U}_a(x-y)+\widetilde{U}_a(x+y)\mright]\d{y}.
%\]
Since $\|K_d\|_{L^1} = 1$ (see~\eqref{eq.vConv}), we may choose $L>0$ large enough so that
\begin{equation}\label{e.c032502}
    \int_L^\infty K_d\leq\norm{U_a}_{L^\infty}\inv\theta.
\end{equation}
Then, since $\max_{x\geq0}\widetilde{U}_a(x)=\theta$ (see~\eqref{eq.SlabModel}), we have, for all $x\geq L$,
\[
	\begin{split}
		V_a(x)
			&= \int_{-\infty}^{\infty} K_d(y) \widetilde{U}_a(x-y) \d y
			= \int_x^\infty K_d(y)\widetilde{U}_a(x-y) \d y
				+ \int_{-\infty}^x K_d(y) \widetilde{U}_a(x-y) \d y
			\\&
			\leq  \int_x^\infty K_d(y) \|U_a\|_{L^\infty} \d y
				+ \int_{-\infty}^x K_d(y) \theta \d y
			\leq \int_L^\infty K_d(y) \|U_a\|_{L^\infty} \d y
				+ \theta
			\leq 2\theta.
	\end{split}
\]
%\[
%	\begin{split}
%		V_a(x)
%			&= \int_{-\infty}^{\infty} K_d(y) \widetilde{U}_a(x-y) \d y
%			\\&
%			= \int_x^\infty K_d(y)\widetilde{U}_a(x-y) \d y
%				+ \int_{-\infty}^x K_d(y) \widetilde{U}_a(x-y) \d y
%			\\&
%			\leq  \int_x^\infty K_d(y) \|U_a\|_{L^\infty} \d y
%				+ \int_{-\infty}^x K_d(y) \theta \d y
%			\\&
%			\leq \int_L^\infty K_d(y) \|U_a\|_{L^\infty} \d y
%				+ \theta
%			\leq 2\theta.
%	\end{split}
%\]
In the last step, we used~\eqref{e.c032502}.
%\begin{align*}
%    \va(x)
%    &= \int_
%    
%    
%    
%    \int_0^\infty K_d(y)\mleft[\widetilde{U}_a(x-y)+\widetilde{U}_a(x+y)\mright]\d{y}\\
%    &=
%    \int_0^x K_d(y)\widetilde{U}_a(x-y)\d y + \int_x^\infty K_d(y)\widetilde{U}_a(x-y)\d y + \int_0^\infty K_d(y)\widetilde{U}_a(x+y)\d y\\
%    &\leq
%    \theta\int_0^x K_d(y)\d{y}
%    +\norm{U_a}_{L^\infty}\int_x^\infty K_d(y)\d{y}
%    +\theta\int_0^\infty K_d(y)\d{y}
%    \leq 3\theta.
%\end{align*}
The remaining inequality in \eqref{eq.v v'<theta} follows by applying the bound on $\abs{\va'}$ from \eqref{eq.PtWiseVaBound}.  This concludes the proof.
\end{proof}

\begin{lemma}[Lower bound on the speed]\label{p.lowerboundspeed}
If $(c_a,U_a,\va)$ is a solution to \eqref{eq.SlabModel} with $c_a \geq 0$, then, for all $\e>0$, there exists $a_\e>0$ and $\theta_\e>0$ such that, for all $a>a_\e$ and for all $\theta\in(0,\theta_\e)$, 
\begin{equation}\label{eq.ca>2-e}
    c_a\geq 2-\e.
\end{equation}
The parameters $a_\eps$ and $\theta_\eps$ depend only on $\|U_a\|_{L^\infty}$.
\end{lemma}
\begin{proof}
We argue by contradiction. Suppose $\e>0$ and $c_a<2-\e$. Note that it is enough to prove the estimate under the assumption that $\eps \in (0,1)$. %the desired estimate \eqref{eq.ca>2-e} holds for $\e\geq2$ if it has already been established for some $\e\in(0,2)$, so we safely assume $\e\in(0,2)$.

Fix $\theta \in (0,\theta_\eps)$ for $\theta_\eps$ to be determined, and let $L$ be as in \Cref{l.v v'<theta}.  Let 
\begin{equation}\label{e.c032601}
	R = \frac{a-1 - L}{2}.
\end{equation}
For $A,\lambda>0$ to be chosen, let
\begin{equation}\label{e.c040501}
    \beta_A(x)
    =
    \frac{1}{A}e^{-\lambda x}\cos\mleft(\frac{\pi}{2R}(x-L-R)\mright)^2
    \qquad
    \text{for }x\in[L,L+2R] = [L, a-1].
\end{equation}

Since $U_a$ is positive on $[L,L+2R]$, we have, by continuity, that $\beta_A<U_a$ on $[L,L+2R]$ if $A$ is sufficiently large. Thus, the following quantity is well-defined:
\[
    A_0=\inf\{A>0:\beta_A<U_a\text{ on }[L,L+2R]\}.
\]
Let $\beta=\beta_{A_0}$. By continuity, there exists $x_0\in [L,L+2R]$ such that $\beta(x_0)=U_a(x_0)$. Moreover, since $\beta(L)=\beta(L+2R)=0$ while both $U_a(L)$ and $U_a(L+2R) = U_a(a-1)$ are positive, it follows that $x_0\in(L,L+2R)$. Also, note that $U_a-\beta$ is nonnegative and attains a minimum of $0$ at $x_0$. As a result, we have
\begin{equation}\label{02.28.1}
    U_a(x_0)=\beta(x_0),\qquad
    U_a'(x_0)=\beta'(x_0),
    \qquad\text{and}\qquad
    (U_a-\beta)''(x_0)\geq0.
\end{equation}
From \eqref{eq.SlabModel}, \eqref{eq.va''}, and \eqref{02.28.1}, we deduce that, at $x_0$, 
\[
\begin{split}
    0
    &\leq
    (U_a-\beta)''
    =-c_aU_a'+\tau\chi \va'U_a'+\tau\chi U_a\mleft(\frac{\va-U_a}{d}\mright)-U_a(1-U_a)-\beta'' \nonumber\\
    &\leq
    -c_a\beta'+\tau\chi \va' \beta' +\frac{\tau\chi}{d}\va\beta-\beta(1-U_a)-\beta''
    \leq     -c_a\beta'+\tau\chi \va' \beta' +\frac{2\theta \tau\chi}{d}\beta-\beta(1-\theta)-\beta''.
\end{split}
\]
In the last step, we used that $U_a \leq \theta$ on $[0,a]$ (see~\eqref{eq.SlabModel}) and $V_a \leq 2\theta$ on $[L, a]$.  
Then, using the explicit formula~\eqref{e.c040501}, multiplying by $A_0 e^{\lambda x_0}$, and using the shorthand $z_0 = (\pi/2R) (x_0 - L - R)$, we find
\[
	\begin{split}
		0
			&\leq (\tau\chi V_a' -c_a) \mleft( -\lambda \cos(z_0)^2 - \frac{\pi}{R} \cos(z_0) \sin(z_0)\mright)
				+ \mleft(\frac{2\theta \tau\chi}{d} +\theta - 1\mright) \cos(z_0)^2 
			\\&
			\qquad
			- \mleft( \lambda^2 \cos(z_0)^2
				+ \frac{2\lambda \pi}{R} \cos(z_0) \sin(z_0)
				+ \frac{\pi^2}{2R^2} \sin(z_0)^2
				- \frac{\pi^2}{2R^2} \cos(z_0)^2\mright)
			\\&
			=  - \frac{\pi^2}{2R^2} \sin(z_0)^2
				- \mleft(1 + \lambda^2 - \frac{\pi^2}{2R^2} - \theta \Big( \frac{2\tau\chi}{d} + 1\Big) + \lambda (\tau\chi V_a' - c_a) \mright) \cos(z_0)^2
			\\&
			\qquad
				+ \frac{\pi}{R} \mleft( c_a - \tau\chi V_a' - 2\lambda\mright) \cos(z_0)\sin(z_0).
\end{split}
\]
The first two terms are negative, and this fact eventually leads to our contradiction.  The third term's sign, however, is indeterminate, so we choose $\lambda = c_a/2$ in order to make this final term small (recall $\sqrt d |V_a'|, V_a \leq 2\theta$ due to \Cref{l.v v'<theta}).  Then
\[
	\begin{split}
		0
			&\leq - \frac{\pi^2}{2R^2} \sin(z_0)^2
				- \mleft(1 + \frac{c_a^2}{4} - \frac{\pi^2}{2R^2} - \theta \Big( \frac{2\tau\chi}{d} + 1\Big) + \frac{c_a\tau\chi}{2}  V_a' - \frac{c_a^2}{2} \mright) \cos(z_0)^2
			\\&
			\qquad
				- \frac{\pi}{R}\tau \chi V_a'  \cos(z_0)\sin(z_0).
			\\&
			= - \frac{\pi^2}{2R^2} \sin(z_0)^2
				- \mleft(1 - \frac{c_a^2}{4} - \frac{\pi^2}{2R^2} - \theta \Big( \frac{2\tau\chi}{d} + 1\Big) - \frac{c_a\tau\chi}{2} V_a' \mright) \cos(z_0)^2
				- \frac{\pi}{R} \tau\chi V_a'  \cos(z_0)\sin(z_0).
\end{split}
\]
Using first the bounds that $0 \leq c_a < 2-\eps$ and the bounds $\sqrt d |V_a'|, V_a \leq 2\theta$ and then Young's inequality, we find
\[
	\begin{split}
		0
			&\leq  - \frac{\pi^2}{2R^2} \sin(z_0)^2
				- \mleft(1 - \Big(1 - \frac\eps2\Big)^2 - \frac{\pi^2}{2R^2} - \theta \Big( \frac{2\tau\chi}{d} + 1\Big) -  \frac{\tau\chi \theta}{\sqrt d}\mright) \cos(z_0)^2
			\\&
			\qquad - \frac{\pi\tau \chi \theta}{R \sqrt d}|\cos(z_0)\sin(z_0)|
			\\&
			\leq - \mleft(1 - \Big(1 - \frac\eps2\Big)^2 - \frac{\pi^2}{2R^2} - \theta \Big( \frac{2\tau\chi}{d} + 1\Big) -  \frac{\tau\chi \theta}{\sqrt d} - \frac{\tau^2\chi^2 \theta^2}{2 d}\mright) \cos(z_0)^2.
\end{split}
\]
Recall $z_0 \in (-\sfrac\pi2,\sfrac\pi2)$ since $x_0 \in (L, L + 2R)$.  Then, $\cos(z_0) > 0$, and we deduce that the above is negative if $R$ is sufficiently large (which necessitates the largeness of $a$) and $\theta$ is sufficiently small.  This is a contradiction and concludes the proof.
\end{proof}

\subsubsection{Preliminaries on uniformly local \texorpdfstring{$L^p$}{Lᵖ}-spaces}%\currentpdfbookmark{Preliminaries on uniformly local Lᵖ-Spaces}{PreliminariesOnUniformlyLocalLp-Spaces}

We present some preliminary estimates on $\va$ and $\phi$ involving uniformly local $L^p$-spaces.  The next two estimates rely on the usefulness of our version of the uniformly local spaces.

First, we compile the behavior of $\phi$.  Importantly, $\phi'$ and $\phi''$ are ``like'' $\phi$ but much smaller (recall that we eventually choose $\eul \ll 1$).  Also, crucially, $K_d$ and $\phi$ can be compared~\eqref{eq.varphiKdBound}.  The proof is omitted as it is elementary.

\begin{lemma}\label{l.varphiBounds}
Suppose that $\eul \in (0,1/\sqrt d)$.  The function $\phi$, defined in \Cref{d.UnifLocalLp}, satisfies the following: there exists $C>0$ such that, for all $x\in \R$,
\begin{equation}\label{eq.varphiBounds}
	\begin{split}
		&Ce^{-\eul\abs{x}}\leq\phi(x)\leq Ce^{-\eul\abs{x}},
		\qquad
		\abs{\phi(x)}\leq1,
		\\&
		\abs{\phi'(x)}\leq\eul\phi(x),
		\qquad\text{ and } \qquad
		\abs{\phi''(x)}\leq\eul^2\phi(x).
	\end{split}
\end{equation}
Additionally, for any $s$ and $x$,
\begin{equation}\label{eq.varphiKdBound}
    K_d(x - s)
%    \leq
%    \frac{C}{\sqrt d}e^{-\eul\abs{x}}
    \leq
    \frac{C}{\sqrt d}\phi_s(x).
\end{equation}
\end{lemma}
%\begin{proof}[Proofy]
%The inequalities in \eqref{eq.varphiBounds} are elementary and follow by applying properties of convolutions to $\phi$.  On the other hand,~\eqref{eq.varphiKdBound} follows directly from the first bound in~\eqref{...}
%%\eqref{eq.varphiKdBound}, decrease $\eul$ so that $\eul<\frac{1}{2\sqrt{d}}$, then observe that it follows from the first inequality in \eqref{eq.varphiBounds} that
%%\[
%%    K_d(x)
%%    =
%%    \frac{1}{2\sqrt d} e^{-\frac{\abs{x}}{\sqrt{d}}}
%%    \leq
%%    \frac{1}{2\sqrt d} e^{-\eul\abs{x}}
%%    \leq
%%    \frac{C}{\sqrt d} \phi(x).\qedhere
%%\]
%\end{proof}

Finally, we compile one last estimate that we require often in the sequel.  This is one major advantage of our particular form of the uniformly local $L^p$-spaces.  A na\"ive estimate that follows directly from the definition of $K_d$~\eqref{eq.vConv} and of $V_a$~\eqref{eq.vaConv} is
\[
	\|V_a\|_{L^\infty} \leq \|U_a\|_{L^\infty}.
\]
%between $V_a$ and $U_a$ in $L^\infty$ (\Cref{l.VaInftyBounds}) 
\Cref{l.UnifLocBoundOnVa} replaces this, which involves the same regularity of $V_a$ and $U_a$, with an estimate that connects a higher regularity norm of $V_a$ (the $L^\infty$-norm) with a lower regularity norm of $U_a$ (the $\lul^p$-norm).

\begin{lemma}\label{l.UnifLocBoundOnVa}
Suppose that $U_a\geq 0$ is related to $V_a$ by~\eqref{eq.vaConv}.  Then, for $p \in [1,\infty)$,  there is $C>0$ such that, for all $x$,  %If $(c_a,U_a,\va)$ solves \eqref{eq.SlabModel}, then, for all $x\in\R$ and for all $p\in[1,\infty)$,
\begin{equation}\label{eq.Va<|Ua|_lul}
    \va(x)
    \leq
    Cd^{-\frac{1}{2p}}\norm{U_a}_{\lul^p}.
\end{equation}
\end{lemma}
\begin{proof}
By \eqref{eq.varphiKdBound} and H\"{o}lder's inequality, we have %for all $x\in[-a,a]$ and for all $p\in(1,\infty)$, we have
%\begin{align*}
%    \va(x)
%    &=
%    \int \widetilde{U}_a(y)K_d^\frac{1}{p}(x-y)K_d^{\frac{p-1}{p}}(x-y)\d y\\
%    &\leq
%    Cd^{-\frac{1}{2p}}
%    \int \widetilde{U}_a(y)\phi^\frac{1}{p}(x-y)K_d^\frac{p-1}{p}(x-y)\d y\\
%    &\leq
%    Cd^{-\frac{1}{2p}}\mleft(\int \widetilde{U}_a(y)^{p}\phi(x-y)\d y\mright)^\frac{1}{p}\mleft(\int K_d(x-y)\d y\mright)^\frac{p-1}{p}
%    \leq 
%    Cd^{-\frac{1}{2p}}\norm{U_a}_{\lul^{p}}.
%    %\mleft(\int_{-\infty}^\infty K_d^\frac{p+1}{2p}(x-y)\d y\mright)^\frac{p}{p+1}
%    % =
%    % \frac{C}{\frac{1}{2(p+1)}}\norm{U_a}_{\lul^{p+1}}\mleft(\frac{2^\frac{3p-1}{2p}p}{p+1}d^\frac{p-1}{4p}\mright)^\frac{p}{p+1}\\
%    % &\leq
%    % Cd^{-\frac{1}{2(p+1)}}\norm{U_a}_{\lul^{p+1}}.
%    \qedhere
%\end{align*}
\begin{align*}
    \va(x)
    &=
    \int_{-\infty}^{\infty} \widetilde{U}_a(y)K_d(x-y)^\frac{1}{p}K_d(x-y)^{\frac{p-1}{p}}\d y\\
    &\leq
    Cd^{-\frac{1}{2p}}
    \int_{-\infty}^{\infty} \widetilde{U}_a(y)\phi^\frac{1}{p}(x-y)K_d(x-y)^\frac{p-1}{p}\d y\\
    &\leq
    Cd^{-\frac{1}{2p}}\mleft(\int_{-\infty}^{\infty} \widetilde{U}_a(y)^{p}\phi(x-y)\d y\mright)^\frac{1}{p}\mleft(\int_{-\infty}^{\infty} K_d(x-y)\d y\mright)^\frac{p-1}{p}
    \leq 
    Cd^{-\frac{1}{2p}}\norm{U_a}_{\lul^{p}}.
    %\mleft(\int_{-\infty}^\infty K_d^\frac{p+1}{2p}(x-y)\d y\mright)^\frac{p}{p+1}
    % =
    % \frac{C}{\frac{1}{2(p+1)}}\norm{U_a}_{\lul^{p+1}}\mleft(\frac{2^\frac{3p-1}{2p}p}{p+1}d^\frac{p-1}{4p}\mright)^\frac{p}{p+1}\\
    % &\leq
    % Cd^{-\frac{1}{2(p+1)}}\norm{U_a}_{\lul^{p+1}}.
    \qedhere
\end{align*}
\end{proof}

\subsection{Uniform bounds on solutions to the slab problem}

In this section, the key estimate is to obtain $a$-independent bounds on $c_a$ and $\|U_a\|_{L^\infty}$.  The main difficulty is that these two quantities are coupled.

\subsubsection{The main estimate}\label{s.three_lemmas}

We obtain the bound on $U_a$ via a main ``energy estimate'' paired with a standard bound on the size of $c_a$ in terms of the $L^\infty$-norm of the advection.  

We state these two results here, although defer their proofs until \Cref{s.proofs}.  This allows us to show exactly how they piece together to obtain the crucial bound on $c_a$ and $\|U_a\|_{L^\infty}$ (\Cref{p.c_a bound}).

We begin by stating the ``energy estimate.''  Note that we can choose $\eul$ to be small and this allows us to temper the dependence on $c_a$.  It is the crucial estimate on which our construction hinges.  
%The following lemma is analogous to Morrey's inequality. Accordingly, it is the key estimate needed to show that solutions to the slab problem are bounded. 
%Its proof is fairly long and technical, so we postpone it until \Cref{sec.pfofkeyestimate}.

\begin{lemma}\label{l.energy}
Suppose that $(c_a,U_a,\va)$ solves \eqref{eq.SlabModel}. If the positive parameter $\eul$ is sufficiently small, depending only on $d$ and $\chi$, then
\begin{equation}\label{e.c032610}
	\|U_a\|_{\lul^2}^2
		\leq C\mleft(\eul c_a^2 + \frac{1}{\eul}\mright).
\end{equation}
\end{lemma}

We now obtain our second ingredient, which is the bound on $c_a$ in terms of $\norm{V_a}_{L^\infty}$.  The main steps in its proof, which is found in \Cref{ss.c_a<=|U_a|}, is somewhat standard and dates at least back to \cite{BerestyckiLarrouturou}.

\begin{lemma}[Upper bound on the speed]\label{l.c_a<=|U_a|} If $(c_a,U_a,\va)$ solves \eqref{eq.SlabModel} and $a>\ln(1/\theta)$, then
\begin{equation}\label{eq.ca<=|U_a|}
    c_a
    \leq
    2+\mleft(\frac{\chi}{\sqrt{d}}+\frac{\chi}{d}\mright)\norm{V_a}_{L^\infty}.
\end{equation}
\end{lemma}

\subsubsection{Uniform estimates on \texorpdfstring{$c_a$}{cₐ} and \texorpdfstring{$\|U_a\|_{L^\infty}$}{‖Uₐ‖ ͚}}

We now put together our estimates from \Cref{s.three_lemmas} in order to establish uniform bounds on $c_a$ and $\|U_a\|_{L^\infty}$. 

\begin{proposition}\label{p.c_a bound}
If $(c_a,U_a,\va)$ solves \eqref{eq.SlabModel}, then $c_a$ and $U_a$ are bounded uniformly in terms of $d$ and $\chi$.
\end{proposition}
\begin{proof}
We first obtain a bound on the speed $c_a$.  By combining \Cref{l.UnifLocBoundOnVa} and \Cref{l.c_a<=|U_a|}, we have
\[
	c_a
		\leq C(1 + \|U_a\|_{\lul^2}).
\]
Then, squaring the above and applying the bound in \Cref{l.energy}, we find
\[
	c_a^2
		\leq C + C \eul c_a^2 + \frac{C}{\eul}.
\]
The bound for $c_a$ follows after taking $\eul$ sufficiently small and absorbing the $c_a^2$ term on the right into the left.

Using this uniform bound on $c_a$ in \Cref{l.energy}, we deduce that
\begin{equation}\label{e.c040101}
	\|U_a\|_{\lul^2}
		\leq C,
\end{equation}
which, by \Cref{l.VaInftyBounds} and \Cref{l.UnifLocBoundOnVa}, yields
\[
	\|V_a\|_{W^{1,\infty}}
		\leq C.
\]

The rest of the proof essentially follows by classical elliptic regularity theory since $U_a$ satisfies an elliptic equation with bounded coefficients and enjoys a (local) $L^2$-bound.  We are not, however, able to find a reference that we can quote ``out-of-the-box,'' so we prove it directly.

Let us show how to conclude a bound in $(-a+1,a-1)$.  The modifications to handle the boundary behavior are immediate so we omit that case.  Fix any $x_0 \in (-a+1, a-1)$ and let $\psi$ be any nonnegative cut-off function that is $1$ on $(x_0 - 1/2,x+1/2)$ and $0$ outside of $(x_0 - 1, x_0 + 1)$.  Then, multiplying~\eqref{eq.SlabModel} by $\psi^2 U_a$ and integrating by parts, we find
\begin{equation}\label{e.c033102}
	\begin{split}
		\int_{-\infty}^\infty \psi^2 |U_a'|^2
			&= \int_{-\infty}^\infty \mleft(-2 \psi \psi' U_a U_a'
				+ c_a \psi^2 U_aU_a'
				- \chi(V_a' U_a)' \psi^2 U_a
				+ \psi^2 U_a^2 (1-U_a) \mright)
			\\&
			\leq \frac{1}{2} \int_{-\infty}^\infty \psi^2 |U_a'|^2
				+ C \int_{-\infty}^\infty (|\psi'|^2 + \psi^2)\mleft(|U_a|^2\mright)
			\\&
			\leq \frac{1}{2} \int_{-\infty}^\infty \psi^2 |U_a'|^2
				+ C \|U_a\|_{\lul^2}^2
			\leq \frac{1}{2} \int_{-\infty}^\infty \psi^2 |U_a'|^2
				+ C.
	\end{split}
\end{equation}
Above we used the bounds on $c_a$ and $V_a$, Young's inequality, and the fact that
\[
	\|\psi U_a\|_{L^2}
		\leq C\|U_a\|_{\lul^2}
		\leq C.
\]
The first inequality above is follows from the compact support of $\psi$, and the second follows from~\eqref{e.c040101}.

After absorbing the gradient term on the right hand side of~\eqref{e.c033102} into the left, we obtain
\[
	\|U'_a\|_{L^2([x_0-\sfrac12,x_0+\sfrac12])}
		\leq C.
\]
The proof is finished after an application of the Sobolev embedding theorem.
\end{proof}

\subsection{Existence of a solution on the slab}

Having established bounds on $c_a$ and $U_a$ that are independent of $a$, we are ready to show the existence of solutions to the slab problem using Leray-Schauder degree theory. 

\begin{proposition}\label{p.slab_existence}
There exists $\theta_0>0$ such that, for all $\theta\in(0,\theta_0)$ and for all $a>0$ sufficiently large, there exists a solution $(c_a,U_a,\va)$ of \eqref{eq.SlabModel} with $\tau=1$.
\end{proposition}
\begin{proof}
We first note that any solution to~\eqref{eq.SlabModel} $(c_a,U_a, V_a)$ must satisfy ``nice'' bounds.  Indeed, from \Cref{p.U_a>=0}, we know that $U_a$ is positive on $(-a,a)$.  Additionally, from \Cref{p.c_a bound}, $c_a$ and $U_a$ are, respectively, bounded in $\R$ and $L^\infty$ independently of $a$.  Applying elliptic regularity theory and noting the bounds on $V_a$ in \Cref{l.VaInftyBounds}, we find that $U_a$ is bounded in $C^{2,\alpha}$.  It then follows from \Cref{p.lowerboundspeed} that $c_a \geq 1$ for all $a$ sufficiently large (up to decreasing $\theta$).  Although not necessary here, we note that all bounds are independent of $a$ as long as $a$ is sufficiently large.  We summarize this as:
\begin{equation}\label{02.22.1}
    \frac{1}{C_0}
    \leq
    c_a+\norm{U_a}_{C^{2,\alpha}([-a,a])}
    \leq
    C_0,
\end{equation}
where $C_0$ is independent of $a$.

We now seek to apply a Leray-Schauder degree theory argument (see \cite{Zeidler} for the general theory).  For $R > C_0$ to be chosen, let
\[
    \calB
    \coloneqq
    \mleft\{(c_a,U_a)\in[0,R]\times C^{1,\alpha}([-a,a]):\norm{U_a}_{C^{1,\alpha}([-a,a])}\leq R,U_a\geq0\mright\},
\]
and define the operator $S_\tau\colon\calB\to\R\times C^{1,\alpha}([-a,a])$ by
\[
    S_\tau(c_a,U_a)
    =
    \mleft(c_a+\theta-\max_{x\geq0}\overline{U}_a(x),\overline{U}_a\mright),
\]
where, for fixed $U_a$, we define $\overline{U}_a$ to be the unique solution of the linear problem
\[
    \begin{cases}
    -c_a\overline{U}_a'+\tau\chi\mleft(\overline{U}_a\va'\mright)'=\overline{U}_a''+U_a(1-U_a) & \text{in }(-a,a),\\
    \overline{U}_a(-a)=1, \quad \overline{U}_a(a)=0
    \end{cases}
\]
with $\va=K_d*\widetilde{U}_a$. % so this equation is indeed linear in $\overline{U}_a$. 
Note that the existence of a solution to \eqref{eq.SlabModel} is equivalent to $S_\tau$.

Standard results in elliptic regularity theory provide bounds on the $C^{2,\alpha}([-a,a])$ norm of $\overline{U}_a$ which depend solely on $a$ and the $C^{1,\alpha}([-a,a])$ norm of $U_a$. Thus, $S_\tau$ is a compact operator. Moreover, if $a$ is sufficiently large, any fixed point of $S_\tau$ is an element of the interior of $\calB$ by \eqref{02.22.1}. Hence, we have
\begin{equation}\label{e.c033004}
    \deg(\id-S_1,\calB,0)
    =
    \deg(\id-S_0,\calB,0).
\end{equation}
Thus, it suffices to show that $\deg(\id -S_0, \calB,0)\neq 0$.  While this is likely known in the literature as fixed points of $S_0$ are solutions of the Fisher-KPP equation, we are unable to find a clear reference and so outline the details below.

In order to show, this we construct a second homotopy $\calF_\tau: \calB \to \R \times C^{1,\alpha}([-a,a])$.  Let
\[
	\calF_\tau(c, U)
		= \mleft(c + \theta - \max_{x\geq 0} \overline{U}(x), \overline{U}\mright)
\]
where $\overline{U}_a$ is the unique solution of
\begin{equation}\label{e.c033001}
	- c \overline{U}' - \overline{U}'' = \tau U(1-U).
\end{equation}
with the boundary conditions $U(-a) = 1$ and $U(a) = 0$.  We notice that $S_0 = \calF_1$ and $\calF_\tau$ is a homotopy connecting $\calF_0$ and $\calF_1$;  Thus,
\begin{equation}\label{e.c033003}
	 \deg(\id-S_0,\calB,0)
	 	=  \deg(\id-\calF_1,\calB,0)
		= \deg(\id - \calF_0, \calB,0)
\end{equation}
as long as any fixed point $\calF_\tau(c_\tau,U_\tau) = (c_\tau,U_\tau)$ avoids the boundary of $\calB$.  When $\tau = 0$, there is nothing to show because the solution is explicit (see~\eqref{e.c033002}). Otherwise, letting $(c,U)$ be the fixed point, we immediately see, as before, that
\[
	\max_{x\geq 0} U_\tau(x) = \theta
		\quad \text{ and }\quad
	\|U_\tau\|_{L^\infty} = 1.
\]
After proving a bound on $c_\tau$, which we do below, this immediately yields a $C^{1,\alpha}$-bound on $U_\tau$.  
The first follows from the definition of $\calF_\tau$, while the second can be deduce from an easy comparison principle argument.  Further, arguing exactly as in \Cref{l.c_a<=|U_a|}, we see that
\[
	c_\tau \leq 2 < R.
\]
To see that $c>0$, let
\begin{equation}\label{e.c033002}
	U_0(x) = \frac{e^{-c_\tau x} - e^{-c_\tau a}}{e^{c_\tau a} - c^{-c_\tau a}}.
\end{equation}
Then $AU_0$ is the unique solution to~\eqref{e.c033001} with $\tau = 0$ and the boundary conditions $AU_0(-a) = A$ and $AU_0(a) = 0$.  Arguing as in \Cref{p.lowerboundspeed} we can ``raise'' $A$ up until $AU_0$ ``touches'' $U_\tau$ from below.  It is clear that $AU_0$ is a strict subsolution of~\eqref{e.c033001} so $AU_0$ cannot ``touch'' $U_\tau$ from below in the interior or at $x=+a$ (by the Hopf lemma).  By analyzing the boundary at $x=-a$, we see that $A =1$, whence we conclude that
\[
	\frac{1 - e^{-c_\tau a}}{e^{c_\tau a} - e^{-c_\tau a}}
		= U_0(0)
		\leq U_\tau(0)
		\leq \theta.
\]
As $c_\tau a \to 0$, we see that the left hand side tends to $\sfrac12$, which yields a contradiction (recall that $\theta < \sfrac12$, see below~\eqref{eq.SlabModel}).  It follows that $c_\tau > 0$.

%From the above, we see that any fixed point must avoid the boundary of $\calB$, up to further increasing $R$.  Hence, recalling~\eqref{e.c033004} and~\eqref{e.c033003}, we find
%\[
%	\deg(\id - S_0, \calB,0)
%		= \deg(\id - \calF_1,\calB,0)
%		= \deg(\id - \calF_0,\calB,0).
%\]
To finish, we need now show that the rightmost term in~\eqref{e.c033003} is non-zero.  Since a fixed point $(c_0,U_0)$ of $\calF_0$ corresponds to the unique solution of
\[
	- c_0 U_0' - U_0'' = 0
	\qquad \text{ where } U_0(-a) = 1, \quad U_0(0) = \theta, ~~ \text{ and } ~~U_0(a) = 0,
\]
we have that
\begin{equation}\label{e.c033005}
	\deg(\id - \calF_0,\calB,0)
		= \pm 1
\end{equation}
as long as we show that the Fr\'echet derivative of $\calF_0$ at the fixed point $(c_0,U_0)$ does not have 1 as an eigenvalue (see~\cite[Proposition~14.5]{Zeidler}).

We establish~\eqref{e.c033005} now.  For each $c$, let $u_c$ be the unique solution of
\[
	-cu_c' - u_c'' = 0
		\qquad \text{ where }
		u_c(-a) = 1 ~~ \text{ and } ~~u_c(a) = 0
\]
Then,
\[
	\calF_0(c, U) 
	    = \mleft(c + \theta - \max_{x\geq 0} u_c(x), u_c\mright)
		= \mleft(c + \theta - u_c(0), u_c\mright).
\]
Notice that the right hand side has no dependence on $U$.  Observe that
\[
	u_{c_0} = U_0
	\qquad\text{ and }\qquad
	\partial_c u_c(0)|_{c=c_0} = a \frac{2 -  (e^{c_0a} + e^{-c_0a})}{(e^{c_0 a} - e^{-c_0 a})^2}.
\]
It follows that, for any $h \in \R$ and $w \in C^{1,\alpha}([-a,a])$,
\begin{equation}\label{e.c033006}
    \mleft(D\calF_0(c_0,U_0)\mright)(h,w)
		= \mleft(\mleft(1+ a \frac{e^{c_0a} + e^{-c_0a} - 2}{(e^{c_0 a} - e^{-c_0 a})^2}\mright), \partial_c u_c|_{c=c_0}\mright)h,
\end{equation}
%\begin{equation}\label{e.c033006}
%	\mleft(D\calF_0(c_0,U_0)\mright)(h,w)
%		= \Big( \Big(1+ a \frac{e^{c_0a} + e^{-c_0a} - 2}{(e^{c_0 a} - e^{-c_0 a})^2}\Big), \partial_c u_c|_{c=c_0} \Big)h,
%\end{equation}
which clearly does not have $1$ as an eigenvalue.  
Thus,~\eqref{e.c033005} is justified and the proof is complete.
%
%However, any fixed point of $S_0$ is a solution to the ``slab" problem for the Fisher-KPP equation
%\[
%	- c_a \overline{U}_a' = \overline{U}_a'' + \overline{U}_a(1-\overline{U}_a).
%\]
%That the degree of $S_0$ is 
%The uniqueness of such solutions is well-known. As a result,
%\[
%    \deg(\id-S_0,\calB,0)=\pm1.
%\]
%In either case, $\deg(\id-S_1,\calB,0)$ is nonzero, thus establishing the existence of a fixed point.
\end{proof}

\section{Obtaining a traveling wave solution: taking \texorpdfstring{$a\to\infty$}{a→∞}}\label{sec.a->infty}

We now use the solution constructed on the slab to obtain a traveling wave solution on the whole space by applying our uniform estimates and a compactness argument.  As usual, a difficulty here is to ensure that, when passing to the limit $a\to\infty$, we do not end up with a trivial traveling wave; i.e., one that is everywhere 1 or everywhere 0.

\begin{proposition}\label{p.existTWS}
There exists a traveling wave solution $(c,U,V)$ in the sense of \Cref{def.TWS}.
\end{proposition}
\begin{proof}
Applying \Cref{p.slab_existence}, we obtain a solution $(c_a,U_a)$ to~\eqref{eq.SlabModel}.  From the estimates  established in the first paragraph of the proof of \Cref{p.slab_existence} as well as elliptic regularity theory,  there is $C_0$ such that
\begin{equation}
	\frac{1}{C_0}
		\leq c_a, \|U_a\|_{C^{2,\alpha}}, \|V_a\|_{C^{2,\alpha}}
		\leq C_0.
\end{equation}
Moreover $U_a > 0$ on $(-a,a)$.

%We first note that the bounds in \eqref{02.22.1} do not depend on $a$. 

By compactness, %we then have that, for all $\theta\in(0,\theta_0)$, 
there exists a sequence $a_n\to\infty$ and $(c,U)\in[\sfrac{1}{C_0},C_0]\times C^{2,\alpha}(\R)$ such that, as $n$ tends to infinity, $c_{a_n}\to c$ in $\R$ and $U_{a_n}\to U$ locally uniformly in $C^2(\R)$. Thus, it follows from \eqref{eq.SlabModel} that $(c,U,V)$ satisfies
\[
    -cU'+\chi(UV')'=U''+U(1-U)\qquad\text{in }\R,
\]
where $V=K_d*U$. 

Next, we prove that $U$ tends to zero as $x\to +\infty$. To do this, we show that $U$ is bounded in $L^1([0,\infty))$. Then, since $U\in C^{0,\alpha}(\R)$, it follows that $U$ vanishes at infinity. 

To establish our $L^1([0,\infty))$ bound on $U$, we first integrate \eqref{eq.SlabModel} over $[0,a]$:
\begin{equation}\label{02.22.2}
    \int_0^a U_a(1-U_a)
    =
    -c_aU_a(0)-U_a'(a)+U_a'(0)-\chi U_a(0)\va'(0).
\end{equation}
It follows from \eqref{02.22.1} that the right-hand side of \eqref{02.22.2} is bounded only in terms of $d$ and $\chi$. Also, due to the normalization $\max_{x\in[0,a]}U_a(x)=\theta<1$ and the positivity of $U_a$ on $[0,a]$, we have
\begin{equation}\label{02.22.3}
    \int_0^a U_a(1-\theta)
    \leq \int_0^a U_a(1-U_a) 
    \leq
    C.
\end{equation}
Taking the limit as $a\to\infty$ and using Fatou's lemma, we find
\[
	\int_0^\infty U
		\leq \liminf_{a\to\infty} \int_0^a U_a 
		\leq C.	
\]
as desired.  The fact that $U(+\infty) = 0$ follows immediately.
%Since $\theta$ depends only $d$ and $\chi$, we conclude that $U_a$ is bounded in $L^1([0,a])$ independently of $a$. This independence allows us to take the limit $a\to\infty$. That is, since $U_{a_n}$ converges locally uniformly to $U$ in $C^2(\R)$, it follows that, up to a subsequence, $\widetilde{U}_{a_n}$ converges in the $L^1([0,\infty))$ norm to $U$; hence, the $L^1([0,\infty))$ norm of $U$ is also bounded by the constant $C_{d,\chi}$ from \eqref{02.22.3}.

Lastly, we show $L:=\liminf_{x\to-\infty}U(x)$ is positive. Choose a sequence $(x_n)$ such that, as $n\to\infty$, $x_n\to-\infty$ and $U(x_n)\to L$, and define 
\[
    U_n(x)\coloneqq\frac{U(x+x_n)}{U(x_n)}\qquad\text{and}\qquad V_n(x)\coloneqq V(x+x_n).
\]
The Harnack inequality shows that $U_n$ is bounded on any compact set containing $x=0$.  Since $(c,U,V)$ satisfies \eqref{eq.TWSmodel}, it follows that $U_n$ and $V_n$ satisfy
\begin{equation}\label{eq.UnVnPDEs}
    \begin{cases}
    -cU_n'+\chi(U_nV_n')'=U_n''+U_n(1-U(x_n)U_n) & \text{in }\R,\\
    -dV_n''=U(x_n)U_n-V_n & \text{in }\R.
    \end{cases}
\end{equation}
The right side of the first equation in \eqref{eq.UnVnPDEs} is bounded by $U_n$. Therefore, by elliptic regularity theory, it follows that, up to a subsequence, $U_n\to U_\infty$ and $V_n\to V_\infty$ in $\Cloc^2(\R)$, and these functions satisfy
\begin{equation}\label{eq.UinftyPDEs}
    \begin{cases}
    -cU_\infty'+\chi(U_\infty V_\infty')'=U_\infty''+U_\infty(1-LU_\infty) & \text{in }\R,\\
    -dV_\infty''=LU_\infty-V_\infty & \text{in }\R.
    \end{cases}
\end{equation}

Suppose, by way of contradiction, that $L=0$. Then, \eqref{eq.UinftyPDEs} becomes
\begin{equation}\label{eq.LzeroUinftyPDEs}
    \begin{cases}
    -cU_\infty'+\chi(U_\infty V_\infty')'=U_\infty''+U_\infty & \text{in }\R,\\
    -dV_\infty''=-V_\infty & \text{in }\R.
    \end{cases}
\end{equation}
Recall that
\begin{equation}\label{e.c040801}
	V_n(x)
		= \int K_d(x+x_n - y) U(y) dy
		= \int K_d(x-y) U(y+x_n) dy.
\end{equation}
Since $U(x_n) \to 0$ as $n\to\infty$, we conclude that $U(y+x_n) \to 0$ locally uniformly in $y$ by the Harnack inequality.  It follows from this and~\eqref{e.c040801} that $V_n \to 0$ locally uniformly in $x$; that is, $V_\infty = 0$. 
%The boundedness of $V$ implies $V_\infty$ is bounded, so the second equation in \eqref{eq.LzeroUinftyPDEs} implies $V_\infty\equiv0$. 
Thus, the first equation in \eqref{eq.UinftyPDEs} becomes
\begin{equation}\label{07.05.1}
    -cU_\infty'=U_\infty''+U_\infty.
\end{equation}
By construction, $U_\infty$ achieves its minimum of $1$ at $x=0$.  From~\eqref{07.05.1}, we see that
%Since $U_\infty\geq0$ and $U_\infty(0)=1$, the maximum principle implies that $U_\infty>0$. Now, suppose $U_\infty$ achieves its minimum at $x_{\min}$, and note that $U_\infty(x_{\min})=1$. Thus, from \eqref{07.05.1}, we have, at $x_{\min}$,
\[
    0= -cU_\infty'
     = U_\infty''+1
     >0.
\]
This is a clearly a contradiction. It follows that $L>0$.  This completes the proof.
\end{proof}

We also have a lower bound on the speed of the traveling wave in \Cref{p.existTWS}.

\begin{proposition}\label{p.speed}
If $(c,U,V)$ is a traveling wave solution to \eqref{eq.model}, then $c\geq2$.
\end{proposition}
\begin{proof}
The proof of this claim follows almost exactly like the proof of \Cref{p.lowerboundspeed}, except we replace each $\theta$ with $\theta_L\coloneqq\min_{x\in[L,\infty)}U(x)$. The proof follows since $\theta_L\to0$ as $L\to\infty$ due to \Cref{def.TWS}.
\end{proof}

The combination of \Cref{p.existTWS} and \Cref{p.speed} yields \Cref{t.thm}.

\section{Proofs of the technical lemmas} \label{s.proofs}

\subsection{The key energy estimate: \texorpdfstring{\Cref{l.energy}}{Lemma \ref{l.energy}}}\label{sec.pfofkeyestimate}

As we saw above, the estimate in \Cref{l.energy} is the centerpiece of the construction of traveling wave solutions.  It is inspired by the proof in \cite{TelloWinkler}; however, being time-independent, it differs in several key aspects.

Before we begin, we prove a state and a prove a small technical estimate that shows we can reduce the $\lul^2$ norm to integrals on a well-chosen domain.
\begin{lemma}\label{l.lul_U_big}
	Suppose that $\|U_a\|_{L^\infty} > 1$.  Define
	\begin{equation}\label{e.c032606}
		x_1 = \min\Big\{x: U_a(x) = \frac12 \mleft(1 + \min\{\|U_a\|_{L^\infty},2\}\mright)\Big\}.
	\end{equation}
	Then
	\begin{equation}
		\|\widetilde{U}_a\|_{\lul^2}^2
			\leq \sup_{s\in\R} \int_{x_1}^a \phi_s U_a^2
				+ \frac{C}{\eul}.
	\end{equation}
\end{lemma}
\begin{proof}
Fix any $\eps>0$.  Choose $s$ such that
\begin{equation}
	\|\widetilde{U}_a\|_{\lul^2}^2
		\leq (1+\eps) \int_{-\infty}^\infty \phi_s \widetilde{U}_a^2.
\end{equation}
Then, using that $\widetilde{U}_a \leq 2$ on $(-\infty,x_1)$ and $\widetilde{U}_a = 0$ on $(a,\infty)$, we find
\begin{equation}
	\begin{split}
	\frac{1}{1+\eps} \|\widetilde{U}_a\|_{\lul^2}^2
		&\leq \int_{-\infty}^{\infty} \phi_s \widetilde{U}_a^2
		\leq 2\int_{-\infty}^{x_1} \phi_s
			+ \int_{x_1}^a \phi_s U_a^2
		\\&
		\leq \frac{C}{\eul}
			+ \int_{x_1}^a \phi_s U_a^2.
	\end{split}
\end{equation}
The proof is concluded by taking $\eps\to 0$.
\end{proof}

Let us comment on the usefulness of $x_1$.  As we see above, the portion of the $\lul^2$-norm occurring on $[x_1,a]^c$ is nicely bounded.  Hence, we do not ``lose anything'' by only integrating on $[x_1,a]$.  On the other hand, in our energy estimate below, we obtain, through integration by parts, boundary terms involving $U_a'$.  By defining $x_1$ as above (and choosing $a$ as the other boundary), we are assured that these have a ``good sign.''

We are now in a position to prove the energy estimate.  A key insight, inspired by the work in \cite{TelloWinkler}, is to  obtain the estimate via the quadratic term $-U_a^2$ in~\eqref{eq.SlabModel}.

\begin{proof}[Proof of \Cref{l.energy}]
We first observe that if $\|U_a\|_{L^\infty} \leq 1$, the proof follows immediately.  Hence, we proceed assuming that $\|U_a\|_{L^\infty} > 1$ so that $x_1$ is well-defined (recall~\eqref{e.c032606}) and we have access to \Cref{l.lul_U_big}.

We claim that, for any $N \geq 1$,
\begin{equation}\label{8.26.13}
\begin{split}
	\sup_{s\in\R} \int_{x_1}^a \phi_s U_a^2
		\leq
		\frac{CN}{\sqrt{\eul}}
		+ CN \eul c_a^2
		+ C\mleft(\sqrt{\eul} + \frac{1}{N} \mright) \|U_a\|_{\lul^2}^2.
\end{split}
\end{equation}
Let us postpone the proof of~\eqref{8.26.13} momentarily and, first, show how to conclude the proof of the lemma with it.

After applying \Cref{l.lul_U_big}, we find that $\|U_a\|_{\lul^2}^2$ satisfies the same inequality:
\begin{equation}\label{e.c040201}
	\|U_a\|_{\lul^2}
		\leq  \frac{CN}{\sqrt{\eul}}
			+ CN \eul c_a^2
			+ C\mleft(\sqrt{\eul} + \frac{1}{N} \mright) \|U_a\|_{\lul^2}^2.
\end{equation}
Indeed, the extra $C/\eul$ term may be absorbed into the $CN/\sqrt{\eul}$ term.  Sufficiently decreasing $\eul$ and increasing $N$, we see that the last term on the right hand side of~\eqref{e.c040201} may be absorbed into the left hand side of~\eqref{e.c040201} to yield the desired inequality:
\[
	\frac{1}{2} \|U_a\|_{\lul^2}^2
		\leq \mleft(1 - C\sqrt{\eul} - \frac{C}{N}\mright) \|U_a\|_{\lul^2}^2
		\leq \frac{CN}{\sqrt{\eul}} + CN \eul c_a^2.
\]

We now show the proof of~\eqref{8.26.13}.  It is enough to establish the result for any fixed $s$.  We begin by rewriting \eqref{eq.SlabModel} using \eqref{eq.va''}: 
\begin{equation}\label{23.04.24.1}
    \mleft(1-\frac{\chi}{d}\mright)U_a^2
    =c_aU_a'-\chi U_a'\va'-\frac{\chi}{d}U_a\va+U_a''+U_a.
\end{equation}
Define $x_1$ as in the statement of \Cref{l.lul_U_big}. 
Multiplying \eqref{23.04.24.1} by $\phi_s$ and integrating over $[x_1,a]$ gives
\begin{equation}\label{8.26.11}
    \mleft(1-\frac{\chi}{d}\mright)\int_{x_1}^{a}\phi_sU_a^2
    =
    c_a\int_{x_1}^{a}\phi_sU_a'
    -\chi\int_{x_1}^{a}\phi_sU_a'\va'
    -\frac{\chi}{d}\int_{x_1}^{a}\phi_sU_a\va
    +\int_{x_1}^{a}\phi_sU_a''
    +\int_{x_1}^{a}\phi_sU_a.
\end{equation}
%\begin{equation}\label{8.26.11}
%\begin{split}
%    \mleft(1-\frac{\chi}{d}\mright)\int_{x_1}^{a}\phi_sU_a^2
%    &=
%    c_a\int_{x_1}^{a}\phi_sU_a'
%    -\chi\int_{x_1}^{a}\phi_sU_a'\va'\\
%    &\qquad
%    -\frac{\chi}{d}\int_{x_1}^{a}\phi_sU_a\va
%    +\int_{x_1}^{a}\phi_sU_a''
%    +\int_{x_1}^{a}\phi_sU_a.
%\end{split}
%\end{equation}
Let $I_1,\dots,I_5$ denote, respectively, the five terms on the right of \eqref{8.26.11}. %Henceforth, $\eul$ is additionally assumed to be in $(0,1)$ until more conditions are applied to it.

\bigskip

\textbf{The integral $\boldsymbol{I_1}$.} Integrating by parts gives
    \begin{equation}
        I_1
        =c_a\phi_s U_a\mleft.\vphantom{\int}\mright|_{x_1}^{a}- c_a\int_{x_1}^{a}\phi_s'U_a
        = - c_a \phi_s(x_1) U_a(x_1) - c_a \int_{x_1}^a \phi_s' U_a
        \leq  - c_a \int_{x_1}^a \phi_s' U_a.
    \end{equation}
    By the properties of $\phi_s$ in \eqref{eq.varphiBounds}, as well as the choice of $x_1$, we have
    \begin{equation*}
	I_1
		\leq 
		C c_a\eul \int_{x_1}^{a}\phi_sU_a
		=
		C\int_{x_1}^{a}\phi_s(\eul c_a)(U_a).
    \end{equation*}
    Young's inequality then gives
    \begin{align*}
	I_1
		&\leq
		CN\int_{x_1}^{a}\phi_s \eul^2c_a^2 + \frac{1}{N} \int_{x_1}^{a} \phi_s U_a^2
		\leq
		CN\eul c_a^2 + \frac{\|U_a\|_{\lul^2}^2}{N}.\\
    \end{align*}
Notice that the nonlinear term $c_a$ appears with a smaller parameter $\eul$ in front, which is what makes this a useful estimate.  On the other hand, the integral term of $U_a$ appears with a small $1/N$ factor, so it can eventually be absorbed.

\bigskip
    
\textbf{The integral $\boldsymbol{I_2}$.} Integrating by parts and using \eqref{eq.va''} gives
    \begin{align}\label{8.26.12}
        I_2 %=-\chi\int_{x_1}^{a}\phi_sU_a^{p-1}U_a'\va'
        &=
        - \chi\int_{x_1}^{a}\phi_s U_a' V_a'
        =
        - \chi\phi_sU_a\va'\mleft.\vphantom{\int}\mright|_{x_1}^{a}
        +\chi\int_{x_1}^{a}\phi_s'U_a\va'
        +\chi\int_{x_1}^{a}\phi_sU_a\va''\nonumber\\
        &=
        -\chi\phi_sU_a\va'\mleft.\vphantom{\int}\mright|_{x_1}^{a}
        +\chi\int_{x_1}^{a}\phi_s'U_a\va'
        +\frac{\chi}{d}\int_{x_1}^{a}\phi_sU_a\va
        -\frac{\chi}{d}\int_{x_1}^{a}\phi_sU_a^2.
    \end{align}
    Consider the first term in the right hand side of~\eqref{8.26.12}.  We use the choice of $x_1$ in~\eqref{e.c032606} to control $\phi_s U_a$ and \Cref{l.VaInftyBounds} to control $V_a'$ in order to obtain%\Cref{l.varphiBounds,l.UnifLocBoundOnVa} to control $V_a'$ in order to obtain:
    \[
        -\chi\phi_sU_a\va'\mleft.\vphantom{\int}\mright|_{x_1}^{a}
        = \chi \phi_s(x_1) U_a(x_1) V_a'(x_1)
        \leq C \|V_a\|_{L^\infty}
        \leq C \|U_a\|_{\lul^2}
        \leq C\sqrt{\eul} \|U_a\|_{\lul^2}^2 + \frac{C}{\sqrt{\eul}}.
%	\leq CA^\frac2{p} + \frac{\|U_a\|_{\lul^2}}{100}.
    \]
    In the second-to-last inequality, we used \Cref{l.UnifLocBoundOnVa}, and in the last one, we used Young's inequality.  Note that there is no importance to the square root above; it is chosen for convenience as the square root naturally shows up in the next estimate~\eqref{e.c040102}.
    
	Arguing using \Cref{l.VaInftyBounds,l.varphiBounds}, we control the second term on the right in~\eqref{8.26.12} as follows:
	%Applying the bounds in \eqref{eq.PtWiseVaBound} and \eqref{eq.Va<|Ua|_lul}, along with Young's inequality, to the first integral in \eqref{8.26.12} gives
	\begin{equation}\label{e.c040102}
	\begin{split}
	\chi\int_{x_1}^{a}\phi_s'U_a\va'
        &\leq
        C \eul  \int_{x_1}^{a}\phi_sU_a \va
        \leq C \eul \int_{x_1}^a \phi_s \mleft( \frac{1}{\sqrt{\eul}} U_a^2 + \sqrt{\eul} \|V_a\|_{L^\infty}^2 \mright)
        \\&= 
        		C \sqrt{\eul} \|U_a\|_{\lul^2}^2
        			+ C\eul \sqrt{\eul} \|V_a\|_{L^\infty}^2\int_{x_1}^a \phi_s.
	\end{split}
	\end{equation}
	The last inequality followed by Young's inequality.  Then, recalling \Cref{l.UnifLocBoundOnVa}, we find
	\[
	\begin{split}
        \chi\int_{x_1}^{a}\phi_s'U_a\va'
		&\leq C \sqrt{\eul} \int_{x_1}^a \phi_s  U_a^2
        			+ C\eul \sqrt{\eul} \|U_a\|_{\lul^2}^2\int \phi_s
		\\&
		\leq C \sqrt{\eul} \int_{x_1}^a \phi_s  U_a^2
			+ C \sqrt{\eul} \|U_a\|_{\lul^2}^2
		\leq C \sqrt{\eul} \|U_a\|_{\lul^2}^2.
	\end{split}
	\]

    Therefore, we have the following estimate for $I_2$:
    \[
        I_2
        \leq
        		\frac{C}{\sqrt{\eul}}
		+ C\sqrt{\eul} \|U_a\|_{\lul^2}^2
		% C_{d,\chi}N\mleft(\frac{1}{\eul}+\eul^p\norm{U_a}^2_{\lul^2}\mright)
	+ \frac{\chi}{d}\int_{x_1}^{a}\phi_sU_a\va
	-\frac{\chi}{d}\int_{x_1}^{a}\phi_sU_a^2.
    \]

\bigskip
    
\textbf{The integral $\boldsymbol{I_3}$.}  We do not estimate $I_3$.  It will be used to cancel one of the bad terms from the estimate of $I_2$. 
%Since $p>1$ and since $d$, $\chi$, $U_a$, $\va$, and $\phi_s$ are nonnegative, we have
%    \[
%        I_3
%        =
%        -\frac{\chi}{d}\int_{x_1}^{a}\phi_sU_a^{p}\va
%        \leq
%        -\frac{\chi}{d}\int_{x_1}^{a}\phi_sU_a^{p}\va.
%    \]

\bigskip
    
\textbf{The integral $\boldsymbol{I_4}$.} Integrating by parts gives
    % \begin{equation}\label{e.03.24.1}
    %     I_4
    %     =
    %     \int_{x_1}^{a}\phi_sU_a^{p-1}U_a''
    %     =
    %     \phi_sU_a^{p-1}U_a'\mleft.\vphantom{\int}\mright|_{x_1}^{a}-\int_{x_1}^{a}\phi_s'U_a^{p-1}U_a'-(p-1)\int_{x_1}^{a}\phi_s^2U_a^{p-2}\abs{U_a'}^2
    % \end{equation}
    \begin{equation}\label{e.c032611}
    \begin{split}
        I_4
       & =
        \int_{x_1}^{a}\phi_sU_a^{p-1}U_a''
        =
        \phi_s U_a'\mleft.\vphantom{\int}\mright|_{x_1}^{a}-\int_{x_1}^{a}\phi_s'
       \leq -\int_{x_1}^{a}\phi_s'
       = \phi_s(x_1) - \phi_s(a)
       \leq 1.
    \end{split}
    \end{equation}
    The first inequality above follows from the choice of $x_1$, which ensures that $U_a'(a) \leq 0$ and $U_a'(x_1) \geq 0$.  We note that we have no control of $\|U_a\|_{L^\infty}$ at this point in the proof, so it is crucial to our argument that $U_a'$ satisfies these inequalities at the endpoints.  This is the motivation for our definition~\eqref{e.c032606} of $x_1$.

\bigskip
    
\textbf{The integral $\boldsymbol{I_5}$.}  By the properties of $\phi_s$ in \eqref{eq.varphiBounds} and Young's inequality, we have
    \[
        I_5
        =
        \int_{x_1}^{a}\phi_sU_a
        \leq
        %\int_{x_1}^{a}\phi_s\mleft(\frac{N}2+\frac{pU_a^2}{N(p+1)}\mright)
        \int_{x_1}^{a}\phi_s\mleft(N+\frac{U_a^2}{N}\mright)
        \leq
        C\frac{N}{\eul}+\frac{1}{N}\int_{x_1}^{a}\phi_sU_a^2
        \leq \frac{CN}{\eul} + \frac{1}{N} \|U_a\|_{\lul^2}^2.
    \]
    
    \bigskip
    
Combining all above with \eqref{8.26.11} yields~\eqref{8.26.13}, which completes the proof of the lemma.
%
%Recall $p\in\mleft(1,\frac{\chi/d}{\chi/d-1}\mright)$, and note that this ensures $1-\frac{\chi}{d}+\frac{\chi}{d}>0$. Next, choose $N>1$ large enough so that $\frac{4}{N}<1-\frac{\chi}{d}+\frac{\chi}{d}$.  Finally, if $\eul$ is sufficiently small, we also see that the final term in~\eqref{8.26.13} is negative.  We, thus, find
%\begin{equation}\label{8.26.14}
%    A c_a
%    	+ \int_{x_1}^{a}\phi_sU_a^2
%	+\int_{x_1}^{a}\phi_sU_a^{p-2}\abs{U_a'}^2
%    \leq
%    C \mleft(\frac{1}{\eul}+\eul^pc_a^2
%    	%+\eul^p\norm{U_a}^2_{\lul^2}
%	+ A \|V_a\|_{L^\infty}
%	+ c_a \eul A\mright).
%\end{equation}
%This completes the proof after noting that $|(U_a^{p/2})'|^2 = (p^2/4) U_a^{p-2} |U_a'|^2$.
\end{proof}

\subsection{An upper bound on \texorpdfstring{$\boldsymbol{c_a}$}{cₐ} by \texorpdfstring{$\boldsymbol{\|V_a\|_{L^\infty}}$}{‖Vₐ‖ ͚ }}
\label{ss.c_a<=|U_a|}

As we noted above, the proof in the sequel is quite standard.

\begin{proof}[Proof of \Cref{l.c_a<=|U_a|}]
Let $\calA=\{A\in\R:\alpha_A\geq U_a\}$, where we define
\[
	\alpha_A(x)=Ae^{-x}.
\]
Observe that $\calA$ is nonempty as $A=e^{a}\norm{U_a}_{L^\infty}$ is in $\calA$. Indeed, with this choice for $A$, we have, for all $x\in[-a,a]$,
\[
    \alpha_A(x)\geq\alpha_A(a)=\norm{U_a}_{L^\infty}\geq U_a(x),
\]
where the first inequality follows since $\alpha_A$ is decreasing. Also, notice that the value $A=\frac{1}{2}e^{-a}$ is positive, but $A\not\in\calA$ since
\[
    \alpha_A(-a)=\frac{1}{2}<1=U_a(-a).
\]
Hence, $\calA$ has a lower bound that is positive. Therefore, by the continuity of $U_a$ and $\alpha_A$, it follows that
\[
    A_0\coloneqq\min\calA
\]
is a well-defined, positive constant.

For simplicity, denote $\alpha=\alpha_{A_0}$. By continuity, there exists $x_0\in[-a,a]$ such that $\alpha(x_0)=U_a(x_0)$. Since $\alpha(a)>0$ but $U_a(a)=0$, we have $x_0\neq a$. Now, suppose $x_0=-a$. Then,
\[
    1=U_a(-a)=\alpha(-a)=A_0e^{a},
\]
implying $A_0=e^{-a}$. Let $x_\theta$ be a point in $[0,a]$ in which $U_a$ achieves its maximum on this interval, i.e., $U_a(x_\theta)=\max_{x\geq0}U_a(x)=\theta$.  Thus, %if we l then, since $\alpha$ is decreasing, we have
\[
    e^{-a}
    =
    A_0
    =
    \alpha(0)
    \geq
    \alpha(x_\theta)
    \geq
    U_a(x_\theta)
    =
    \theta.
\]
This contradicts our assumptions that $a>\ln(1/\theta)$. Hence, $x_0\neq -a$.

We deduce that $x_0\in(-a,a)$. The function $\alpha-U_a$ attains its minimum of zero at $x_0\in(-a,a)$, which yields the following:
\[
    U_a(x_0)=\alpha(x_0),
    \qquad
    U_a'(x_0)=\alpha'(x_0)=-\alpha(x_0),
    \qquad\text{and}\qquad
    U_a''(x_0)\leq\alpha''(x_0)=\alpha(x_0).
\]
Using these relations in~\eqref{eq.SlabModel}, along with the bounds on $\va'$ and $\va''$ from \Cref{l.VaInftyBounds} and the positivity of $U_a$ from \Cref{p.U_a>=0}, we have, at $x_0$,
\begin{align}\label{06.29.1}
    0
    &=
    U_a''+U_a(1-U_a)-\tau\chi U_a\va''-\tau\chi U_a\va'+c_aU_a' \nonumber\\
    &\leq
    \alpha+\alpha+\tau\chi\alpha\frac{1}{d}\norm{V_a}_{L^\infty}+\tau\chi\alpha\frac{1}{\sqrt{d}}\norm{V_a}_{L^\infty}-c_a\alpha \nonumber\\
    &\leq\alpha\mleft(2+\frac{\chi}{d}\norm{V_a}_{L^\infty}+\frac{\chi}{\sqrt{d}}\norm{V_a}_{L^\infty}-c_a\mright).
\end{align}
Since $\alpha(x_0)$ is positive, its coefficient in \eqref{06.29.1} must be nonnegative. Hence, \eqref{eq.ca<=|U_a|} follows.
\end{proof}

\bigskip

\noindent{\bf Acknowledgments.} CH was supported by NSF grants DMS-2003110 and DMS-2204615.  MR was supported by NSF grant GCR-2020915.  The authors acknowledge support of the Institut Henri Poincar\'e (UAR 839 CNRS-Sorbonne Université), and LabEx CARMIN (ANR-10-LABX-59-01).  The authors thank Quentin Griette for helpful discussions regarding the numerical work in~\cite{fu2021sharp}.

\bibliographystyle{plain}
\bibliography{ref.bib}

\end{document}